\documentclass[a4paper, 11pt]{amsart}
\usepackage[utf8]{inputenc}
\usepackage[english]{babel}
\usepackage[T1]{fontenc}
\usepackage{amsmath, amssymb, amsfonts}
\usepackage[all]{xy}
\usepackage{enumerate}
\usepackage{graphicx}
\usepackage{geometry}
\usepackage{rotate}
\usepackage{MnSymbol}
\usepackage{comment}
\usepackage{hyperref}

\title{On the representation theory of partition (easy) quantum groups}
\author{Amaury Freslon}
\author{Moritz Weber}
\address{Univ. Paris VII, IMJ--PRG, B\^{a}timent Sophie Germain, Case 7012, 75205 Paris Cedex 13, France}
\address{Saarland University, Fachbereich Mathematik, Postfach 151159, 66041 Saarbr\"ucken, Germany}
\email{freslon@math.jussieu.fr, weber@math.uni-sb.de}
\date{\today}

\theoremstyle{plain}
\newtheorem{thm}{Theorem}[section]
\newtheorem{prop}[thm]{Proposition}

\newtheorem{lem}[thm]{Lemma}
\newtheorem*{conj}{Conjecture}
\newtheorem*{q}{Question}
\newtheorem*{philo}{Tannakanian philosophy}

\theoremstyle{definition}
\newtheorem{de}[thm]{Definition}
\newtheorem{ex}[thm]{Example}

\theoremstyle{remark}
\newtheorem{rem}[thm]{Remark}

\DeclareMathOperator{\Aut}{Aut}
\DeclareMathOperator{\Id}{Id}
\DeclareMathOperator{\id}{id}
\DeclareMathOperator{\Ima}{Im}
\DeclareMathOperator{\Irr}{Irr}
\DeclareMathOperator{\Hom}{Hom}

\DeclareMathOperator{\Proj}{Proj_{\CC}}
\DeclareMathOperator{\Projc}{Proj_{\CC^{\circ, \bullet}}}
\DeclareMathOperator{\ProjP}{Proj}
\DeclareMathOperator{\rl}{rl}
\DeclareMathOperator{\Span}{Span}
\DeclareMathOperator{\Sym}{Sym_{\CC}}
\DeclareMathOperator{\Symc}{Sym_{\CC^{\circ, \bullet}}}

\newcommand{\C}{\mathbb{C}}
\newcommand{\CC}{\mathcal{C}}
\newcommand{\D}{\Delta}

\newcommand{\G}{\mathbb{G}}

\newcommand{\N}{\mathbb{N}}
\newcommand{\Pbp}{P_{\textnormal{bp}}}

\newcommand{\R}{\mathbb{R}}

\newcommand{\U}{\mathcal{U}}
\newcommand{\Z}{\mathbb{Z}}

\newcommand{\ii}{\imath}

\newcommand{\singleton}{\uparrow}
\newcommand{\vierpart}{
\mathrel{\offinterlineskip
\hskip0ex\hbox{$\sqcap$}\hskip -.4ex\hbox{$\sqcap$} \hskip -0.4ex\hbox{$\sqcap$}}}
\newcommand{\idpart}{|}
\newcommand{\paarpart}{\sqcap}
\newcommand{\baarpartbaustein}{\rotatebox{180}{$\sqcap$}}
\newcommand{\baarpart}{
\mathrel{\vcenter{\offinterlineskip \hbox{$\baarpartbaustein$}}}}
\newcommand{\vierpartrot}{
\mathrel{\vcenter{\offinterlineskip
\hbox{$\baarpart$} \vskip -.1ex \hbox{\hskip .5ex $\shortmid$} \vskip -.1ex \hbox{$\paarpart$}}}}
\newcommand{\doublepairrot}{
\mathrel{\vcenter{\offinterlineskip
\hbox{$\baarpart$} \vskip +.7ex  \hbox{$\paarpart$}}}}
\newcommand{\doublesingletonrot}{
\mathrel{\vcenter{\offinterlineskip
\hbox{$\shortmid$} \vskip +.7ex \hbox{$\shortmid$}}}}
\newcommand{\crosspart}{
\mathrel{\offinterlineskip
\hbox{$/$}\hskip -.95ex\hbox{$\backslash$}}}
\newcommand{\midmid}{
\mathrel{\vcenter{\offinterlineskip
\hbox{$\shortmid$} \vskip -1.6ex \hbox{$\shortmid$}}}}
\newcommand{\halflibpart}{
\mathrel{\offinterlineskip
\hbox{$\bigtimes$}\hskip -1.55ex\hbox{$\midmid$}}}

\usepackage{calc}
\newcounter{PartitionDepth}
\newcounter{PartitionLength}
\newcommand{\parti}[2]{
 \begin{picture}(#2,#1)
 \setcounter{PartitionDepth}{-1-#1}
 \put(#2,\thePartitionDepth){\line(0,1){#1}}
 \end{picture}}
\newcommand{\partii}[3]{
 \begin{picture}(#3,#1)
 \setcounter{PartitionLength}{#3-#2}
 \setcounter{PartitionDepth}{-1-#1}
 \put(#2,\thePartitionDepth){\line(0,1){#1}}     
 \put(#3,\thePartitionDepth){\line(0,1){#1}}
 \put(#2,\thePartitionDepth){\line(1,0){\thePartitionLength}}
 \end{picture}}
\newcommand{\partiii}[4]{
 \begin{picture}(#4,#1)
 \setcounter{PartitionLength}{#4-#2}
 \setcounter{PartitionDepth}{-1-#1}
 \put(#2,\thePartitionDepth){\line(0,1){#1}}
 \put(#3,\thePartitionDepth){\line(0,1){#1}}
 \put(#4,\thePartitionDepth){\line(0,1){#1}}
 \put(#2,\thePartitionDepth){\line(1,0){\thePartitionLength}} 
 \end{picture}}

\newcommand{\upparti}[2]{
 \begin{picture}(#2,#1)
 \setcounter{PartitionDepth}{#1}
 \put(#2,0){\line(0,1){#1}}
 \end{picture}}
\newcommand{\uppartii}[3]{
 \begin{picture}(#3,#1)
 \setcounter{PartitionLength}{#3-#2}
 \setcounter{PartitionDepth}{#1}
 \put(#2,0){\line(0,1){#1}}     
 \put(#3,0){\line(0,1){#1}}
 \put(#2,\thePartitionDepth){\line(1,0){\thePartitionLength}}
 \end{picture}}
\newcommand{\uppartiii}[4]{
 \begin{picture}(#4,#1)
 \setcounter{PartitionLength}{#4-#2}
 \setcounter{PartitionDepth}{#1}
 \put(#2,0){\line(0,1){#1}}
 \put(#3,0){\line(0,1){#1}}
 \put(#4,0){\line(0,1){#1}}
 \put(#2,\thePartitionDepth){\line(1,0){\thePartitionLength}} 
 \end{picture}}

\renewcommand{\leq}{\leqslant}
\renewcommand{\geq}{\geqslant}

\begin{document}

\begin{abstract}
Compact matrix quantum groups are strongly determined by their intertwiner spaces, due to a result by S.L. Woronowicz. In the case of \emph{easy quantum groups} (also called \emph{partition quantum groups}), the intertwiner spaces are given by the combinatorics of partitions, see the inital work of T. Banica and R. Speicher. The philosophy is that all quantum algebraic properties of these objects should be visible in their combinatorial data.
We show that this is the case for their fusion rules (i.e. for their representation theory). As a byproduct, we obtain a unified approach to the fusion rules of the quantum permutation group $S_N^+$, the free orthogonal quantum group $O_N^+$ as well as the hyperoctahedral quantum group $H_N^+$.
We then extend our work to unitary easy quantum groups and link it with a "freeness conjecture" of T. Banica and R. Vergnioux.
\end{abstract}

\maketitle

\section{Introduction}

In 1937, R. Brauer developed in \cite{brauer1937algebras} a combinatorial tool, called \emph{Brauer diagrams}, to study the representation theory of the orthogonal group $O_{N}$. A Brauer diagram is a partition in pairs of a set of $2k$ points and the basic idea is to associate to each of these partitions an endomorphism of the vector space $(\C^{N})^{\otimes k}$. This construction produces intertwiners between tensor powers of the fundamental representation of $O_{N}$ and one of the main results of \cite{brauer1937algebras} is that any intertwiner can be recovered as a linear combination of the "combinatorial ones". This means that Brauer diagrams encode in some sense the whole representation theory of $O_{N}$. Easy quantum groups can be seen as a wide generalization of the correspondance between the compact group $O_{N}$ and the set of Brauer diagrams.

Brauer algebras (algebras generated by Brauer diagrams) can be defined over any field and have been extensively studied from the algebraic point of view. When the field is $\C$, some refinements of the construction yield  the Fuss-Catalan algebras introduced by D. Bisch and V. Jones in \cite{bisch1997algebras} to describe the combinatorics of intermediate subfactors. This is where the link with compact quantum groups began to manifest itself. This link was made clear by the founding works of T. Banica \cite{banica1996theorie}, \cite{banica1997groupe}, \cite{banica1999symmetries} and \cite{banica2002quantum}. In those papers, he used various versions of the Temperley-Lieb algebra to compute the representation theory of the free quantum groups of S. Wang and A. van Daele \cite{wang1995free}, \cite{van1996universal}. Another example is the paper \cite{banica2009fusion} by T. Banica and R. Vergnioux where the fusion rules of the quantum reflexion groups are computed using operators associated with colored partitions.

Building on this background, T. Banica and R. Speicher gave in \cite{banica2009liberation} a very general setting for the study of "partition quantum groups". The procedure is as follows: First choose a set of partitions of $k + l$ points for any pair of integers $k$ and $l$, then build vector spaces $\Hom(k, l)$ by taking the linear span of the operators associated to the partitions of $k+l$.  If now the set of partitions satifies some stability properties (if it is a \emph{category of partitions}), the Tannaka-Krein duality theorem of S.L. Woronowicz \cite{woronowicz1988tannaka} asserts the existence of a unique compact quantum group $\G$ together with a fundamental representation $u$ such that for any $k$ and $l$,
\begin{equation*}
\Hom(k, l) = \Hom(u^{\otimes k}, u^{\otimes l}).
\end{equation*}
Such quantum groups $\G$ are said to be \emph{easy quantum groups}. Note that they are also called \emph{partition quantum groups}, even though the name "easy" has now become standard in the field.

The aforementioned procedure can be recast in a purely categorical framework: The data of the sets $\Hom(k, l)$ yield a unique \emph{concrete complete monoidal W*-category}. Our aim is to understand the simple objects of this category and determine their fusion rules. Consequently, our work only uses some elementary combinatorics and linear algebra. We should emphasize that no technical knowledge on compact quantum groups is needed in the sequel.

Among easy quantum groups are  some classical groups which have been classified in \cite{banica2009liberation}: We have of course the orthogonal group $O_{N}$, furthermore the symmetric group $S_{N}$, the hyperoctahedral group $H_{N}$, the bistochastic group $B_{N}$ and some symmetrized versions of them. The other easy quantum groups can be divided into several classes, the most important one being that of \emph{free quantum groups}, which have been classified in \cite{banica2009liberation} and \cite{weber2012classification}. It consists of "liberated versions" of the classical easy groups: $O_{N}^{+}$, $S_{N}^{+}$, $H_{N}^{+}$ and $B_{N}^{+}$ and some (possibly freely) symmetrized versions. The classification of the remaining easy quantum groups has been done in  \cite{banica2010classification}, \cite{weber2012classification}, \cite{raum2013easy} and \cite{raum2013full}.

We will endeavour in this paper a comprehensive study of the representation theory of easy quantum groups. The main result is the collection of Theorems 
\ref{ThmUnitaryEquivalence}, \ref{PropGeneralDecompositionRefined}  and \ref{ThmFusionRules}. 
They describe a family of unitary \emph{"combinatorial"} representations of an arbitrary easy quantum group $\G$ with remarkable properties:
\begin{itemize}
\item Their definition is very simple and directly based on the partitions given by the "easy" structure: To every projective partition (i.e. a partition which is symmetric in an appropriate sense) we assign a representation of $\G$. Its most important datum is the number of through-blocks of the associated partition.
\item They form a decomposition of all tensor powers of the fundamental representation of $\G$. This implies that any irreducible representation appears as a summand of at least one of them.
\item There is a very simple combinatorial characterization of unitary equivalence.
\item They are "stable under tensor products": Tensor products of these representations again decompose using only combinatorial ones.
\end{itemize}
In some sense, these representations form a "combinatorial subgroup" of the discrete quantum dual of $\G$.

The question of irreducibility for these combinatorial representations is adressed but not solved explicitely. This leads directly into the core of the "group issue": Easy groups are very complicated to describe from the combinatorial point of view. For example, it is well known that representations of the symmetric group $S_{N}$ are indexed by Young diagrams, but it is quite unclear how Young diagrams can be used to decompose our combinatorial representations. This "group issue" will not be adressed here, but we will try to suggest some links between our work and classical representation theory enlightening the problem.

Nevertheless, the quantum world offers various objects with quite different behaviours. The whole strength of Theorems 
\ref{ThmUnitaryEquivalence}, \ref{PropGeneralDecompositionRefined}  and \ref{ThmFusionRules}. 
appears in the context of free easy quantum groups. Giving a unified and rather simple treatment of their representation theory is the first application of our work.

Orthogonal easy quantum groups are interesting objects which have been widely studied and are completely classified as already mentioned. The world of \emph{unitary} easy quantum groups is, on the contrary, still mysterious (see a forthcoming paper \cite{tarrago2015unitary}). However, extending our results to this context (where one has to deal with \emph{colored} partitions) is straightforward. This gives a systematic and efficient way to study the representation theory of unitary easy quantum groups and in particular of unitary \emph{free} easy quantum groups. This leads us to an interesting open problem stated by T. Banica and R. Vergnioux in \cite{banica2009fusion}, which we named the \emph{freeness conjecture}. This conjecture asserts that there is a very strong link between the structure of the category of partitions defining an easy quantum group and the algebraic structure of its fusion semiring. The way we recover the representation theory of free quantum groups in this paper not only gives evidence for this conjecture, but even enables us to make it more precise by giving a candidate for the generators of the "free" structure of the fusion semiring. We believe that this is an important step towards a proof of the freeness conjecture.

Let us now outline the organization of the paper. Section \ref{SecPartitions} deals with the combinatorial machinery behind our approach. In particular, we give a canonical way to decompose partitions and show the usefulness of it in the study of the so-called \emph{projective partitions} which will prove crucial in the sequel. After this, we give in Section \ref{SecQuantumGroups} some basic definitions and facts concerning the theory of compact quantum groups as introduced by S.L. Woronowicz. We also introduce easy quantum groups, our main object of study. In Section \ref{SecRepresentation} we turn to the main results of this paper. We build a family of representations of an arbitrary easy quantum group $\G$ out of the projective partitions of its category of partitions. We give some criteria for irreducibility and unitary equivalence and explain how to compute their tensor products -- the fusion rules with respect to partitions. The short section \ref{SecSpecial} details two extreme cases: classical groups and free quantum groups. In the first case, we try to show the difficulty of the problem by linking it with purely algebraic issues. In the second case, our method gives a unified and simple way to recover the known fusion rules for $S_N^+, O_N^+$ and $H_N^+$. The exposition there is written in such a way that a reader mainly interested in this aspect may jump directly to this point and explore the article from there. Eventually, we adress the issue of \emph{unitary quantum groups} in Section \ref{SecUnitary}. We explain how to extend our results to this setting and discuss the "freeness conjecture" of T. Banica and R. Vergnioux, trying to decide when the fusion semiring of a free unitary quantum group is free.

\section*{Acknowledgments}

This work was initiated during a stay of the first author at the University of Saarbrücken. He whishes to thank R. Speicher and his team for their kind hospitality, as well as U. Franz and \emph{Campus France (Egide)} for making this stay possible. The second author thanks University Paris VII for a stay he made there in the Operator Algebras team where this article was completed. Both authors also wish to the referee for interesting comments and suggestions on this work.

\section{Partitions and associated linear maps}\label{SecPartitions}

\subsection{Partitions and some basic notions}

We first introduce the main combinatorial tool of this paper. A \emph{partition} $p$ is a combinatorial object given by $k\in \N_{0}$ upper and $l\in \N_{0}$ lower points which may be connected by some strings. This gives rise to a partition of the ordered set on $k+l$ points, with an additional information: which points are upper and which are lower. When useful, we label the upper points by $\{1, \dots, k\}$ from left to right and likewise $\{1', \dots, l'\}$ for the lower points.
As an example, we consider the following partitions $p_{1}$ and $p_{2}$.
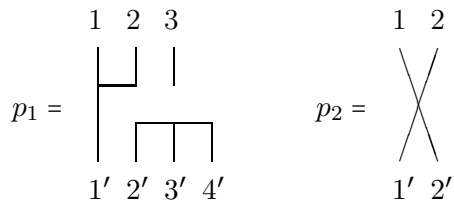
\setlength{\unitlength}{0.5cm}
\begin{center}
\begin{picture}(15,6)
 \put(0,2.2){$p_{1} =$}
 \put(1,5){\parti{3}{1}}
 \put(1,5){\partii{1}{1}{2}}
 \put(1,5){\parti{1}{3}}
 \put(1,1){\uppartiii{1}{2}{3}{4}}
 \put(2,4.5){1}
 \put(3,4.5){2}
 \put(4,4.5){3}
 \put(2,0){$1'$}
 \put(3,0){$2'$}
 \put(4,0){$3'$}
 \put(5,0){$4'$}
 \put(8,2.2){$p_{2} =$}
 \put(10,4.5){1}
 \put(11,4.5){2}
 \put(10,0){$1'$}
 \put(11,0){$2'$}
 \put(10.2,1){\line(1,3){1}}
 \put(10.2,4){\line(1,-3){1}}
\end{picture}
\end{center}
A set $V$ of points connected by a string in $p$ is called a \emph{block} and we write
\begin{equation*}
p = \{V_{1}, \dots, V_{r}\},
\end{equation*}
if $p$ consists of the blocks $V_{1}, \dots, V_{r}$. The number of blocks in a partition $p$ is denoted by $b(p)$. A block which consists only of a single point is a \emph{singleton}. Blocks containing upper points as well as lower points are called \emph{through-blocks} and their number is denoted by $t(p)$. For convenience, we set
\begin{equation*}
\beta(p) = b(p) - t(p),
\end{equation*}
the number of \emph{non-through-blocks}. For example, the partition $p_{1}$ in the above example consists of a through-block connecting the points $1$, $2$ and $1'$, a non-through-block on the points $2'$, $3'$ and $4'$, and finally a singleton (which is always a non-through-block) on the point 3. Thus, $b(p_{1}) = 3$, $t(p_{1}) = 1$ and $\beta(p_{1}) = 2$. 

The set of all partitions is denoted by $P(k, l)$, for $k, l\in \N_{0}$. If $k = l = 0$, then $P(0, 0)$ consists only of the \emph{empty partition}  $\emptyset$. The collection of all sets $P(k, l)$ is denoted by $P$. If all blocks of a partition $p$ consist of exactly two points, the partition is called a \emph{pair partition}. The set of all pair partitions on $k$ upper and $l$ lower points is denoted by $P_{2}(k, l)$, and likewise the collection of all pair partitions is denoted by $P_{2}$. If the connecting strings of a partition $p\in P(k, l)$ do not cross, the partition is called \emph{noncrossing}, and we denote by $NC(k, l)$ (resp. $NC_{2}(k, l)$) the set of all noncrossing partitions (resp. all noncrossing pair partitions); likewise $NC$ and $NC_{2}$. The above partition $p_{1}\in NC(3, 4)$ is non-crossing whereas $p_{2}\in P(2, 2)$ is crossing (it consists of the block connecting 1 and $2'$ and a second block on 2 and $1'$). Furthermore, $p_{2}$ is a pair partition while $p_{1}$ is not.

In the sequel, we will decompose partitions in a certain way. For this, we need two subclasses of partitions.

\begin{de}\label{DefBuildingPartition}
A partition $p\in P(k, l)$ is a \emph{building partition}, if
\begin{enumerate}
\item All lower points of $p$ are in different blocks.
\item For any lower point $1'\leqslant x'\leqslant l'$ of $p$, there exists at least one upper point which is connected to it and we define $\min_{\text{up}}(x')$ to be the smallest upper point $1\leqslant y\leqslant k$ that is connected to $x'$.
\item For any two lower points $1'\leqslant a' < b'\leqslant l'$ of $p$, we have $\min_{\text{up}}(a') < \min_{\text{up}}(b')$.
\end{enumerate}
We denote by $\Pbp(k, l)$ the set of all building partitions in $P(k, l)$.
\end{de}

\begin{rem}
If $p\in\Pbp(k, l)$ is a building partition, then its number $t(p)$ of through-blocks is equal to $l$.
\end{rem}

The following partition $p_{3}$ is an example of a building partition (note that the points $3$ and $5$ are \emph{not} connected to $4$, $6$ and $2'$; likewise $11$, $13$, $15$ and $4'$ form one block and $12$, $14$ and $5'$ another).
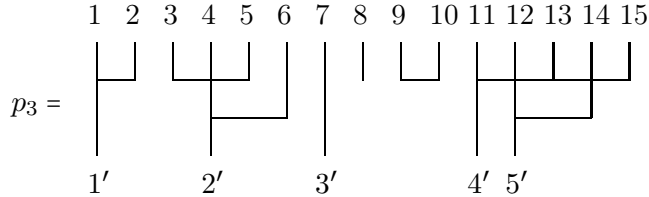
\setlength{\unitlength}{0.5cm}
\begin{center}
\begin{picture}(18,6)
 \put(0,2.2){$p_{3} =$}
 \put(1,5){\parti{3}{1}}
 \put(1,5){\partii{1}{1}{2}}
 \put(1,5){\partii{1}{3}{5}}
 \put(1,5){\partii{2}{4}{6}}
 \put(1,5){\parti{3}{4}}
 \put(1,5){\parti{3}{7}}
 \put(1,5){\parti{1}{8}}
 \put(1,5){\partii{1}{9}{10}}
 \put(1,5){\partiii{1}{11}{13}{15}}
 \put(1,5){\parti{3}{11}}
 \put(1,5){\partii{2}{12}{14}}
 \put(1,5){\parti{3}{12}}
 \put(2,4.5){1}
 \put(3,4.5){2}
 \put(4,4.5){3}
 \put(5,4.5){4}
 \put(6,4.5){5}
 \put(7,4.5){6}
 \put(8,4.5){7}
 \put(9,4.5){8}
 \put(10,4.5){9}
 \put(11,4.5){10}
 \put(12,4.5){11}
 \put(13,4.5){12}
 \put(14,4.5){13}
 \put(15,4.5){14}
 \put(16,4.5){15}
 \put(2,0){$1'$}
 \put(5,0){$2'$}
 \put(8,0){$3'$}
 \put(12,0){$4'$}
 \put(13,0){$5'$}
\end{picture}
\end{center}

\begin{de}
A partition $p\in P_{2}(k, k)$ is a \emph{through-partition}, if all blocks of $p$ are through-blocks. The partition $p$ hence consists of pairs that connect exactly one upper point to exactly one lower point.
\end{de}

\begin{rem}
Note the following two facts:
\begin{enumerate}
\item If $p\in P_{2}(k, k)$ is a through-partition, then $t(p) = b(p) = k$.
\item Through-partitions in $P_{2}(k, k)$ correspond in a natural way to permutations in $S_{k}$.
\end{enumerate}
\end{rem}

\subsection{Operations on partitions}

There are several operations on the set $P$ of partitions:
\begin{itemize}
\item  The \emph{tensor product} of two partitions $p\in P(k, l)$ and $q\in P(k', l')$ is the partition $p\otimes q\in P(k+k', l+l')$ obtained by \emph{horizontal concatenation}, i.e. the first $k$ of the $k+k'$ upper points are connected by $p$ to the first $l$ of the $l+l'$ lower points, whereas $q$ connects the remaining $k'$ upper points with the remaining $l'$ lower points.
\item The \emph{composition} of two partitions $p\in P(k, l)$ and $q\in P(l, m)$ is the partition $qp\in P(k, m)$ obtained by \emph{vertical concatenation}. Connect $k$ upper points by $p$ to $l$ middle points and then continue the lines by $q$ to $m$ lower points. This yields a partition, connecting $k$ upper points with $m$ lower points. By the composition procedure, certain loops might appear resulting from blocks around the middle points. More precisely, consider the set $L$ of elements in $\{1, \dots, l\}$ which are not connected to an upper point of $p$ nor to a lower point of $q$. The lower row of $p$ and the upper row of $q$ both induce partitions of the set $L$. The maximum (with respect to inclusion) of these two partitions is the \emph{loop partition} of $L$, its blocks are called \emph{loops} and their number is denoted $\rl(q, p)$. To finish the operation, we remove all the middle points (and in particular all the loops) in order to produce a partition in $P(k, m)$.
\item The \emph{involution} of a partition $p\in P(k, l)$ is the partition $p^{*}\in P(l, k)$ obtained by turning $p$ upside down.
\item We also have a \emph{rotation} on partitions. Let $p\in P(k, l)$ be a partition connecting $k$ upper points with $l$ lower points. Shifting the very left upper point to the left of the lower points (or the converse) -- without changing the strings connecting the points -- gives rise to a partition in $P(k-1, l+1)$ (or in $P(k+1, l-1)$), called a \emph{rotated version} of $p$. This procedure may also be performed on the right-hand side of the $k$ upper and $l$ lower points. In particular, for a partition $p\in P(0, l)$, we might rotate the very left point to the very right and vice-versa.
\end{itemize}

These operations (tensor product, composition, involution and rotation) are called the \emph{category operations}. By $\idpart\in P(1,1)$ we denote the \emph{identity partition}, connecting the upper point to the lower point.

\begin{de}\label{DeCatPart}
A collection $\CC$ of subsets $\CC(k, l)\subseteq P(k, l)$ (for every $k, l\in\N_{0}$) is a \emph{category of partitions} if it is invariant under the category operations and if the identity partition $\idpart\in P(1, 1)$ is in $\CC(1, 1)$.
\end{de}

Examples of categories of partitions include $P$, $P_{2}$, $NC$ and $NC_{2}$.

\begin{rem}\label{RemCatOpAndBlockNumber}
The relations between the category operations and the number of blocks (resp. through-blocks) are the following:
\begin{enumerate}
\item For  $p\in P(k, l)$ and $q\in P(k', l')$ we have the formul\ae{}
\begin{equation*}
b(p\otimes q) = b(p)+b(q)\text{ and }t(p\otimes q) = t(p)+t(q)
\end{equation*}
for the blocks (resp. the through-blocks) and likewise $\beta(p\otimes q) = \beta(p)+\beta(q)$ for the non-through-blocks (recall that $\beta(p) = b(p)-t(p)$).
\item The composition is associative and for any three partitions $p, q, r\in P$ we have the formula (as soon as it makes sense)
\begin{equation*}
\rl(p,q)+\rl(pq,r) = \rl(p,qr)+\rl(q,r)
\end{equation*}
corresponding to $(pq)r = p(qr)$. Furthermore, we have $t(pq)\leqslant \min(t(p), t(q))$. 
\item Note that in general, there is no way to compute $b(pq)$ from the numbers $b(p)$, $b(q)$, $t(p)$, $t(q)$ and $\rl(p, q)$. For instance, the four partitions 
\begin{equation*}
\begin{array}{cc}
p_{1} = \left\{\{1, 2\}, \{3, 4, 3', 4'\}, \{1', 2'\}\right\}, & q_{1} = \left\{\{1, 2\}, \{3, 4, 1', 4'\}, \{2', 3'\}\right\} \\
p_{2} = \left\{\{1, 2, 3', 4'\}, \{3, 4\}, \{1', 2'\}\right\}, & q_{2} = \left\{\{1, 2\}, \{3, 4, 3', 4'\}, \{1', 2'\}\right\}
\end{array}
\end{equation*}
satisfy $b(p_{1}q_{1})\neq b(p_{2}q_{2})$ but have all other numbers in common.
\item For any two partitions $p, q\in P$, we have $(pq)^{*} = q^{*}p^{*}$ and
\begin{equation*}
\rl(p, q) = \rl(q^{*}, p^{*}).
\end{equation*}
Furthermore $b(p^{*}) = b(p)$, $t(p^{*}) = t(p)$ and $\beta(p^{*}) = \beta(p)$.
\end{enumerate}
\end{rem}

\subsection{Decomposition of partitions}

The following notion of a projective partition will be essential in the sequel.

\begin{de}
A partition $p\in P(k, k)$ is said to be
\begin{enumerate}
\item \emph{Symmetric} if $p^{*} = p$.
\item \emph{Idempotent} if $pp = p$.
\item \emph{Projective} if $p$ is symmetric and idempotent.
\end{enumerate}
\end{de}

The following partitions $p_{4}$ and $p_{5}$ are projective partitions, whereas $p_{2}$ of the above example is not (it is not idempotent).
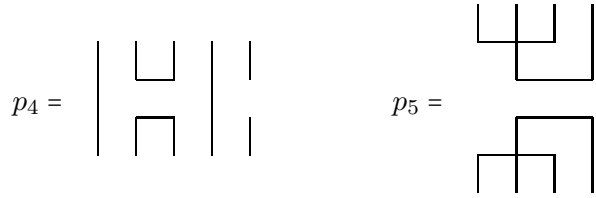
\setlength{\unitlength}{0.5cm}
\begin{center}
\begin{picture}(15,6)
 \put(0,2.2){$p_4 =$}
 \put(1,5){\parti{3}{1}}
 \put(1,5){\partii{1}{2}{3}}
 \put(1,1){\uppartii{1}{2}{3}}
 \put(1,5){\parti{3}{4}}
 \put(1,5){\parti{1}{5}}
 \put(1,1){\upparti{1}{5}}
 \put(10,2.2){$p_5 =$}
 \put(11,6){\partii{1}{1}{3}}
 \put(11,6){\partii{2}{2}{4}}
 \put(11,0){\uppartii{1}{1}{3}}
 \put(11,0){\uppartii{2}{2}{4}}
\end{picture}
\end{center}
Any projective partition $p$ can in particular be written as $p = r^{*}r$. The non-trivial fact which will prove crucial is that the converse also holds: \emph{Any} partition $r\in P(k, l)$ gives rise to a projective partition $r^{*}r\in P(k, k)$. This will be one of the outcomes of a special decomposition of partitions that we call the \emph{through-block decomposition}. Before proving it, we gather some computations in a lemma.

\begin{lem} \label{LemPPStar}
The following hold:
\begin{enumerate}
\item If $p\in P_{2}(k, k)$ is a through-partition, then $p^{*}p = pp^{*} = \idpart^{\otimes k}$ and $\rl(p, p^{*}) = \rl(p^{*}, p) = 0$.
\item If $p\in \Pbp(k, l)$ is a building partition, then $pp^{*} = \idpart^{\otimes l}$ and $\rl(p, p^{*}) = b(p)-t(p) = \beta(p)$. 
\item If $p\in \Pbp(k, l)$ is a building partition, then $p^{*}p$ is a projective partition with $t(p^{*}p) = l$ and $\rl(p^{*}, p) = 0$.
\end{enumerate}
\end{lem}

\begin{proof}
To prove $(1)$, just notice that, seen as permutations, the partition $p^{*}$ is the inverse of $p$.
 
By point $(2)$ of Definition \ref{DefBuildingPartition}, each of the upper points of $p^{*}$ is connected to a middle point in the procedure of the composition $pp^{*}$, which in turn is connected to exactly one lower point by $p$. Point $(1)$ of Definition \ref{DefBuildingPartition} thus yields $pp^{*} = \idpart^{\otimes l}$. The number $b(p)-t(p)$ counts the number of blocks in $p$ consisting only of upper points of $p$, i.e. blocks giving rise to loops which are removed in $pp^{*}$. This concludes the proof of $(2)$.
 
The proof of $(3)$ uses the same ideas. The partition $p^{*}p$ is symmetric, and $p^{*}pp^{*}p = p^{*}p$ by part $(2)$. Moreover, no loop arises in the composition $p^{*}p$ since all middle points are connected both to upper and lower points.
\end{proof}

\begin{prop}\label{PropThroughBlockDecomp} 
Let $p\in P(k,l)$ be a partition with $t(p)$ through-blocks. Then, there is a unique triple of partitions $(q, r, s)$, called the \emph{through-block decomposition} of $p$, such that
\begin{itemize}
\item $p = q^{*}rs$
\item $r\in P_2(t(p), t(p))$ is a through-partition
\item $s\in \Pbp(k, t(p))$ and $q\in \Pbp(l, t(p))$ are building partitions.
\end{itemize}
Furthermore, we have the following equalities:
\begin{itemize}
\item $\rl(r, s) = \rl(q^{*}, rs) = 0$
\item $b(p) = b(q) + b(s) - t(p)$
\item $\beta(p) = \beta(q) + \beta(s)$
\end{itemize}
\end{prop}

We will often use the notations $p = p_{l}^{*}p_{m}p_{u}$ for the elements of the through-block decomposition of $p$.

\begin{proof}
Definition \ref{DefBuildingPartition} yields a recipe for defining the partition $s\in P(k, t(p))$: we restrict the partition $p$ to its $k$ upper points and connect each of its through-blocks to exactly one of the $t(p)$ lower points. Here, we respect the order as in point $(3)$ of Definition \ref{DefBuildingPartition}. Thus, we obtain a building partition $s\in \Pbp(k, t(p))$. An analoguous procedure yields the partition $q\in \Pbp(l, t(p))$. Now, there is a unique way to connect the $t(p)$ lower points of $s$ to the $t(p)$ upper points of $q^{*}$ such that the composition yields $p$. By this, we obtain the through-partition $r\in P_{2}(t(p), t(p))$ such that $p = q^{*}rs$.
 
Let $p = (q')^{*}r's'$ be another decomposition into building partitions $s'\in \Pbp(k, m)$ and $q'\in \Pbp(l, m)$ and a through-partition $r'\in P_{2}(m, m)$. Then, $m = t((q')^{*}r's') = t(p)$. Furthermore, two upper points $a$ and $b$ of $s'$ are in the same block if and only if they are so in $p$. Indeed, if they are not connected by $s'$, they are not connected in $r's'$ since $r'$ is a through-partition. Now, all of the upper points of $(q')^{*}$ are in different blocks by definition, thus $a$ and $b$ are not connected in $(q')^{*}(r's')$. Conversely,  if two upper points $a$ and $b$ of $s'$ are in the same block, they are in the same block in $(q')^{*}r's'$ as well. We infer that the partitions $p$ and $s'$ coincide on their upper points. Furthermore, the partition $s'$ connects exactly those upper points to lower points which belong to through-blocks in $p$. By Definition \ref{DefBuildingPartition}, this can only be done in a unique way, which yields $s = s'$. Likewise, we deduce $q = q'$. Using point $(2)$ of Lemma \ref{LemPPStar}, we obtain
\begin{equation*}
r' = q'(q')^{*}r's'(s')^{*} = q'p(s')^{*} = qps^{*} = qq^{*}rss^{*} = r.
\end{equation*}
The partition $s$ consists of $t(s) = t(p)$ through-blocks and $\beta(s) = b(s)-t(p)$ upper blocks (i.e. blocks which contain only upper points). Since $p$ and $s$ coincide on their upper points, $p$ has exactly $\beta(s)$ upper blocks. In the same way, we deduce that $p$ has $\beta(q) = b(q^*)-t(p)$ lower blocks. Summing up, $p$ has 
\begin{equation*}
b(p) = \beta(s)+\beta(q)+t(p) = (b(s)-t(p))+(b(q)-t(p))+t(p) = b(s)+b(q)-t(p)
\end{equation*}
blocks. 
\end{proof}

\begin{rem}
If $\CC$ is a category of partitions and $p\in\CC$, then $q$, $r$ and $s$ \emph{need not} belong to $\CC$ in general.
\end{rem}

Let us now give some elementary properties of this through-block decomposition.

\begin{lem}\label{LemPPStarProj}
Let $p\in P(k, l)$ be a partition and let $p = q^{*}rs$ be its through-block decomposition according to Proposition \ref{PropThroughBlockDecomp}. Then,
\begin{enumerate}
\item We have $p^{*}p = s^{*}s$ and $\rl(p^{*}, p) = \rl(q, q^{*}) = \beta(q)$.
\item We have $pp^{*} = q^{*}q$ and $\rl(p, p^{*}) = \rl(s, s^{*}) = \beta(s)$.
\item The partitions $p^{*}p$ and $pp^{*}$ are projective and $t(p^{*}p) = t(pp^{*}) = t(p)$.
\end{enumerate}
Furthermore, $pp^{*}p = p$ and $\beta(p) = \rl(p^{*}, p) + \rl(p, p^{*})$.
\end{lem}

\begin{proof}
To prove $(1)$, we compute (using Lemma \ref{LemPPStar}) 
\begin{equation*}
p^{*}p = s^{*}r^{*}qq^{*}rs = s^{*}r^{*}rs = s^{*}s.
\end{equation*}
The number $\rl(p^*,p)$ of removed loops when composing $p^*$ and $p$ is given by the number of non-through-blocks in the lower points of $p$. This number equals $\beta(q)$, the number of non-through-blocks of $q$. Since furthermore $\rl(r^{*}, r) = 0$ and $\rl(s^{*}, s) = 0$ by Lemma \ref{LemPPStar}, we get $\rl(p^{*}, p) = \rl(q, q^{*})$. 

The proof of $(2)$ follows from the fact that, by uniqueness, $s^{*}r^{*}q$ is the through-block decomposition of $p^{*}$.

Eventually, the partition $s^{*}s$ is projective and $t(s^{*}s) = t(p)$ by Lemma \ref{LemPPStar}. This and $(1)$ prove $(3)$. Moreover, $pp^{*}p = (q^{*}rs)(s^{*}r^{*}q)(q^{*}rs) = q^{*}rs = p$ by Lemma \ref{LemPPStar}.
\end{proof}

As a direct consequence, we have the announced criterion to build projective partitions.

\begin{prop}\label{PropProjectiveQstarQ}
A partition $p\in P(k, k)$ is projective if and only if there exists a partition $q\in P(k, l)$ such that $p = q^{*}q$.
\end{prop}

Let us now detail the through-block decomposition in the particular case of a symmetric partition.
\begin{itemize}
\item If $p\in P(k, k)$ is symmetric, then its through-block decomposition is of the form $p = s^{*}rs$ with $r = r^{*}$. This follows from the uniqueness of the decomposition $p = q^{*}rs$ and the equality $p = p^{*} = s^{*}r^{*}q$.
\item If $p\in P(k, k)$ is symmetric, then it is projective if and only if $r = \idpart^{\otimes t(p)}$. Indeed, $pp = s^{*}rss^{*}rs = s^{*}s$, by Lemma \ref{LemPPStar}, and $p = s^{*}rs$. The result follows from the uniqueness of the through-block decomposition.
\item The only noncrossing through-partition $r\in NC_{2}(m,m)$ is $r = \idpart^{\otimes m}$, thus the through-block decomposition of any non-crossing partition $p\in NC(k, l)$ is of the form $p = q^{*}s$, where $q, s\in NC$. Hence, if $p\in NC(k, k)$ is a noncrossing partition, then $p$ is projective if and only if it is symmetric. 
\end{itemize}

\subsection{Linear maps associated to partitions}

We can relate partitions to linear maps on tensor powers of $\C^{N}$ for any integer $N\in \N_{0}$, if we fix a basis $e_{1}, \dots, e_{N}$ of $\C^{N}$. For a partition $p\in P(k, l)$ we define the linear map
\begin{equation*}
\mathring{T}_{p}:(\C^{N})^{\otimes k} \mapsto (\C^{N})^{\otimes l}
\end{equation*}
by the following formula from \cite[Def 1.6]{banica2009liberation}:
\begin{equation*}
\mathring{T}_{p}(e_{i_{1}} \otimes \dots \otimes e_{i_{k}}) = \sum_{j_{1}, \dots, j_{l} = 1}^{n} \delta_{p}(i, j)e_{j_{1}} \otimes \dots \otimes e_{j_{l}},
\end{equation*}
where $\delta_{p}(i, j) = 1$ if and only if all strings of the partition $p$ connect equal indices of the multi-index $i = (i_{1}, \dots, i_{k})$ in the upper row with equal indices of the multi-index $j = (j_{1}, \dots, j_{l})$ in the lower row. Otherwise, $\delta_{p}(i, j) = 0$.

\begin{rem}
Note that $\mathring{T}_{p} = \mathring{T}_{q}$ implies $p = q$ as soon as $N \geqslant 2$.
\end{rem}

The interplay between the category operations on partitions and the assignment $p \mapsto \mathring{T}_{p}$ was studied by T. Banica and R. Speicher in \cite[Prop. 1.9]{banica2009liberation}. It can be summarized as follows.

\begin{prop}\label{PropTpRules}
The assignment $p\mapsto \mathring{T}_{p}$ satisfies
\begin{enumerate}
\item $\mathring{T}_{p^{*}} = \mathring{T}_{p}^{*}$
\item $\mathring{T}_{p\otimes q} = \mathring{T}_{p}\otimes \mathring{T}_{q}$
\item $\mathring{T}_{pq} = N^{-\rl(p, q)}\mathring{T}_{p}\mathring{T}_{q}$.
\end{enumerate}
\end{prop}

It will prove convenient in the sequel to normalize the operators $\mathring{T}_{p}$.

\begin{de}\label{DefTp}
Let $p\in P(k, l)$ be a partition, we set
\begin{equation*}
T_{p} = N^{-\frac{1}{2}\beta(p)}\mathring T_{p}.
\end{equation*}
\end{de}

Let us first give the relation between the rank of these operators and the properties of the partitions.

\begin{prop}\label{PropRankTp}
The rank of $T_{p}$ is $N^{t(p)}$, for any partition $p\in P(k, l)$.
\end{prop}

\begin{proof}
Upper non-through-blocks in $p$ yield equations defining the kernel of $T_{p}$ and have consequently no influence on the rank. If $p\in P(k, l)$, the image of $T_{p}$ is a subspace of $(\C^{N})^{\otimes l}$. Each lower non-through-block implies that some tensor factors reduce to a one-dimensional subspace and each through-block collapses all the tensor factors which are in it to one copy of $\C^{N}$. Hence, the image is the tensor product of one copy of $\C$ for each lower non-through-block and one copy of $\C^{N}$ for each through-block, i.e. the rank of $T_{p}$ is $N^{t(p)}$.
\end{proof}

Before explaining the advantages of this normalization, let us gather some computations.

\begin{lem}\label{LemGamma}
Let $q\in P(k, l)$ and $p\in P(l, m)$ be two partitions and set
\begin{equation*}
\gamma(p, q) = \frac{1}{2}(\beta(p)+\beta(q)-\beta(pq))-\rl(p,q)
\end{equation*}
Then, $\gamma(p, q)$ vanishes in the two following cases.
\begin{enumerate}
\item For every partition $p$, $\gamma(p, p^{*}) = \gamma(p, p^{*}p) = 0$. In particular, if $p$ is a projective partition, then $\gamma(p, p) = 0$.
\item If $p$ and $q$ are two projective partitions such that $pq = q$ then $\gamma(q, p) = 0$.
\end{enumerate}
\end{lem}

\begin{proof}
To prove (1), first note that  $\gamma(p^{*}, pp^{*}) = \frac{1}{2}\beta(pp^{*}) - \rl(p^{*}, pp^{*})$ using $p^{*}pp^{*} = p^{*}$. By Lemma \ref{LemPPStarProj}, we know that $\beta(p) = \rl(p^{*}, p) + \rl(p, p^{*})$, hence $\beta(pp^{*}) = 2\rl(pp^{*}, pp^{*})$. Using the associativity rules of Remark \ref{RemCatOpAndBlockNumber}, we infer
\begin{equation*}
\rl(pp^{*}, pp^{*}) = \rl(p, p^{*}pp^{*}) + \rl(p^{*}, pp^{*}) - \rl(p, p^{*}) = \rl(p^{*}, pp^{*})
\end{equation*}
This yields $\gamma(p^{*}, pp^{*}) = 0$. Furthermore, we have $\rl(p^{*}, pp^{*}) = \rl(p^{*}, p)$ by the following computation, where we use the through-block decomposition $p = p_{u}^{*}p_{m}p_{l}$ and again Remark \ref{RemCatOpAndBlockNumber}:
\begin{align*}
\rl(p^{*}, pp^{*}) & = \rl(p_{l}^{*}p_{m}^{*}p_{u}, p_{u}^{*}p_{u}) \\
& = \rl(p_{l}^{*}p_{m}^{*}, p_{u}p_{u}^{*}p_{u}) + \rl(p_{u}, p_{u}^{*}p_{u}) - \rl(p_{l}^{*}p_{m}^{*}, p_{u}) \\
& = \rl(p_{u}, p_{u}^{*}p_{u}) \\
& = \rl(p_{u}p_{u}^{*}, p_{u}) + \rl(p_{u}, p_{u}^{*}) - \rl(p_{u}^{*}, p_{u})
\end{align*}
Now, $\rl(p_{u}^{*}, p_{u}) = 0$ by Lemma \ref{LemPPStar} and $p_{u}p_{u}^{*} = \idpart^{\otimes t(p)}$, which yields $\rl(p_{u}p_{u}^{*}, p_{u}) = 0$. Hence, $\frac{1}{2}\beta(pp^*)=\rl(pp^{*}, pp^{*}) = \rl(p^{*}, pp^{*}) = \rl(p_{u}, p_{u}^{*}) = \rl(p^{*}, p)$ by Lemma \ref{LemPPStarProj}. Thus,
\begin{equation*}
\gamma(p, p^{*}) = \beta(p) - \frac{1}{2}\beta(pp^{*}) - \rl(p, p^{*}) = \beta(p) - \rl(p^{*}, p) - \rl(p, p^{*}) = 0.
\end{equation*}

To prove $(2)$, we use the formula of point $(2)$ of Remark \ref{RemCatOpAndBlockNumber} to get
\begin{equation*}
\rl(p, p) + \rl(p, q) = \rl(p, pq) + \rl(p, q)
\end{equation*}
Using the equality $pq = q$, we obtain $\rl(p, q) = \rl(p, p)$ and thus
\begin{equation*}
\gamma(p, q) = \frac{1}{2}\beta(p)-\rl(p, p) = 0
\end{equation*}
by Lemma \ref{LemPPStarProj}.
\end{proof}

The following proposition is a restatement of Proposition \ref{PropTpRules}. It shows the advantage of the renormalization from the point of view of operator theory.

\begin{prop}
The assignment $p\mapsto T_{p}$ satisfies:
\begin{enumerate}
\item $T_{p^{*}} = T_{p}^{*}$.
\item $T_{p\otimes q} = T_{p}\otimes T_{q}$
\item $T_{pq} = N^{\gamma(p,q)}T_{p}T_{q}$.
\item If $p\in P(k, k)$ is projective, then $T_{p}$ is a projection.
\item For any $p\in P(k, l)$, the map $T_{p}$ is a partial isometry with $T_{p}T_{p}^{*} = T_{pp^{*}}$ and $T_{p}^{*}T_{p} = T_{p^{*}p}$.
\end{enumerate}
\end{prop}

\begin{proof}
$(1)$ and $(2)$ come from the equalities $\beta(p^{*}) = \beta(p)$ and $\beta(p\otimes q) = \beta(p)+\beta(q)$ (see Remark \ref{RemCatOpAndBlockNumber}) and $(3)$ follows directly from the definition of $T_{p}$ and Proposition \ref{PropTpRules}. Lemma \ref{LemGamma} together with $(1)$ and $(3)$, yield $(4)$ and $(5)$.
\end{proof}

\begin{rem}
Let us make two points:
\begin{itemize}
\item This normalization of the maps $\mathring{T}_{p}$ is the best possible with respect to points $(1)$, $(4)$ and $(5)$ of the last proposition. In fact, if $T'_{p} = N^{-\alpha(p)}\mathring{T}_{p}$ is any normalization of $T_{p}$ (with $\alpha(p)\in \R$) such that the analogues of $(1)$, $(4)$ and $(5)$ hold, let us prove that $\alpha(p) = \frac{1}{2}\beta(p)$. We have:
\begin{equation*}
T'_{p}T'_{q} = N^{-(\alpha(p)+\alpha(q))+\rl(p, q)+\alpha(pq)}T'_{pq}
\end{equation*}
Thus, if $(4)$ holds then $\alpha(p) = \rl(p, p)$ for any projective partition $p$. Furthermore, if $(5)$ holds, then
\begin{equation*}
\alpha(p)+\alpha(p^{*}) = \rl(p, p^{*})+\alpha(pp^{*})
\end{equation*}
for any partition $p\in P(k, l)$. Since $pp^{*}$ is projective, we deduce that $\alpha(pp^{*}) = \rl(pp^{*}, pp^{*}) = \rl(p^{*}, p)$. Using $(1)$ (as well as Lemma \ref{LemPPStarProj}), we obtain
\begin{equation*}
\alpha(p) = \frac{1}{2}(\rl(p, p^{*})+\rl(p^{*}, p))=\frac{1}{2}\beta(p).
\end{equation*}
\item Although we may have $\gamma(p, q) = 0$ in some cases (e.g. in Lemma \ref{LemGamma}), this correction term from the composition of maps $T_{p}$ and $T_{q}$ does not always vanish (see for instance the partitions $p_{1}$ and $q_{1}$ from Remark \ref{RemCatOpAndBlockNumber}).
\end{itemize}
\end{rem}

\subsection{The lattice of projective partitions}

By $\Proj(k)$, we denote the set of all projective partitions $p\in \CC(k, k)$, where $k\in\N$ and $\CC$ is a category of partitions. If $\CC = P$, we simply write $\ProjP(k)$. There is a natural order structure on this set, pulled back from the order structure on the associated projections $T_{p}$.

\begin{de}
Let $k\in\N$ and let $p, q\in \ProjP(k)$ be two projective partitions. We say that $q$ \emph{is dominated by} $p$ (or that $p$ \emph{dominates} $q$) and we write $q\preceq p$, if $pq = q$ -- or equivalenty $qp = q$.
\end{de}

\begin{lem}\label{LemProjAndTP}
Let $p, q\in \ProjP(k)$. Then $q\preceq p$ if and only if $T_{q}\leq T_{p}$ as projections (i.e. $T_{p}T_{q} = T_{q}T_{p} = T_{q}$) for some $N\geqslant 2$.
\end{lem}

\begin{proof}
By Lemma \ref{LemGamma}, we know that $q\preceq p$ implies that $\gamma(p, q) = 0$. Hence,
\begin{equation*}
T_{p}T_{q} = T_{pq} = T_{q}
\end{equation*}
in that case. Conversely, assume $T_{q}\leq T_{p}$. Then,
\begin{equation*}
N^{-\gamma(p, q)}T_{pq} = T_{p}T_{q} = T_{q} = T_{q}T_{p} = N^{-\gamma(q, p)}T_{qp}.
\end{equation*}
Since, by Remark \ref{RemCatOpAndBlockNumber}, $\gamma(s, t) = \gamma(t^{*}, s^{*})$ for \emph{any} two partitions $s$ and $t$, we infer that $T_{pq} = T_{qp}$ and hence $pq = qp$. Thus, $pq$ is a projective partition and we get
\begin{equation*}
N^{-\gamma(p, q)}T_{pq} = T_{q} = T_{q}^{2} = N^{-2\gamma(p, q)}T_{pq}T_{pq} = N^{-2\gamma(p, q)}T_{pq}
\end{equation*}
This yields $\gamma(p, q) = 0$, hence $T_{q} = T_{pq}$ and $pq = q$.
\end{proof}

\begin{rem}
The partial order $\preceq$ on $\ProjP(k)\subset P(k, k)$ is not to be confused with the usual partial order $\leq$ on $P(k, k)$. Recall that for two partitions $p, q\in P(k, l)$, we write $q\leq p$ if each block of $q$ is contained in a block of $p$. Thus, in the partial order $\leq$, a  partition is smaller if it is given by refinement of the block structure. For instance, the partitions
\begin{equation*}
p = \idpart\;\idpart\text{ and }q = \vierpartrot
\end{equation*}
in $P(2, 2)$ satisfy $p\leq q$. For $k\in\N$, the minimum with respect to $\leq$ is the partition in $P(k, k)$ where all blocks are singletons (no points are connected), whereas the maximum is the partition consisting of a single block (all points are connected to each other).

The partial order $\preceq$ does \emph{not} coincide with $\leq$. In fact, it even behaves quite differently. For instance, for the above partitions $p, q\in P(2, 2)$, we have $q\preceq p$ but $q\geq p$. As another, example consider the partitions
\begin{equation*}
r = \doublepairrot\text{ and } s = \doublesingletonrot\;\doublesingletonrot
\end{equation*}
in $P(2, 2)$. Then, $r\geq s$, but $r$ and $s$ are not comparable with respect to $\preceq$. For $k\in\N$, the maximum with respect to $\preceq$ is the partition $\idpart^{\otimes k}$, but there is no minimum.
\end{rem}

We can describe the order $\preceq$ in the following way: A projective partition $q$ is smaller than a projective partition $p$ if it is given by reducing the through-block structure -- either by cutting or by combining through-block strings. Let us make this more precise.

\begin{lem}
Let $p, q\in\ProjP(k)$ be two projective partitions. Then $q\preceq p$ if and only if $q$ is obtained from $p$ by
\begin{itemize}
\item keeping the non-through-block structure, i.e. every non-through-block of $p$ is a non-through-block of $q$ of exactly the same form,
\item leaving through-blocks of $p$ either invariant,
\item or unifying some of them to larger through-blocks,
\item and/or turning them into non-through-blocks.
\end{itemize}
\end{lem}

\begin{proof}
We first study how a partition $t\in P(l, m)$ can operate (by composition) on a building partition $r\in \Pbp(k, l)$ such that the resulting partition $tr$ is a building partition again. First note that $t$ cannot change the non-through-block structure of $r$. Now, let $x$ be an upper point of $t$ (which is hence at the same time a lower point of $r$) and denote by $V_{x}$ the through-block of $r$ connected to $x$.

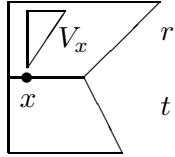
\setlength{\unitlength}{0.5cm}
\begin{figure}[h]
\begin{center}
\begin{picture}(5,4)
 \put(0,0){\line(0,1){4}}
 \put(0,0){\line(1,0){3}}
 \put(0,4){\line(1,0){4}}
 \put(0,2){\line(1,0){2}}
 \put(2,2){\line(1,-2){1}}
 \put(2,2){\line(1,1){2}}
 \put(0.3,1.8){$\bullet$}
 \put(0.3,1.2){$x$}
 \put(4,1){$t$}
 \put(4,3){$r$}
 \put(0.5,2.25){\line(0,1){1.5}}
 \put(0.5,2.25){\line(2,3){1}}
 \put(0.5,3.75){\line(1,0){1}}
 \put(1.3,2.8){$V_{x}$}
\end{picture}
\end{center}
\end{figure}

\emph{Case 1.} If $x$ is a singleton in $t$ (i.e. $x$ is not connected to any other point of $t$), then $t$ turns the through-block $V_{x}$ of $r$ into a non-through-block.

\emph{Case 2.} Now let $x$ not be a singleton in $t$. First, note that it can be connected to at most one lower point of $t$ -- otherwise $tr$ would not be a building partition. Hence, if $x$ is not connected to any other upper point of $t$, it is connected to exactly one lower point, which effects that $t$ leaves the through-block $V_{x}$ of $r$ invariant. On the other hand, if $x$ is connected to some other upper points of $t$, the corresponding through-blocks of $r$ are connected to a single block by $t$. Depending on whether $x$ is also connected to a lower point of $t$ or not, the resulting block either is a through-block or not. 

We conclude that $t$ operates on $r$ by composition in the following way:
\begin{itemize}
\item Every non-through-block of $r$ remains of the same form (in $tr$).
\item Through-blocks of $r$ are either left invariant,
\item or are unified with other through-blocks to larger through-blocks,
\item and/or they are turned into non-through-blocks.
\end{itemize}
Conversely, if a building partition $s$ results from another building partition $r$ by the above operations, this can be modelled by a partition $t$ and the composition $tr = s$.

Now, write $q = s^{*}s$ and $p = r^{*}r$ in the through-block decomposition (by Proposition \ref{PropThroughBlockDecomp}), where $s$ and $r$ are building partitions. Then, $qp = s^{*}sr^{*}r = s^{*}tr$ where $t = sr^{*}$ is a partition in $P(t(p), t(q))$. Hence, $qp = q$ if and only if $tr = s$ (note that $qp = q$ implies $sqp = sq$ and we have $ss^{*} = \idpart^{\otimes t(q)}$ by Lemma \ref{LemPPStar}). Thus, if $qp = q$, then $s$ results from $r$ by the above operations. On the other hand, if $s$ results from $r$ by the above operations, then this can be modelled by a partition $t'$ such that $t'r = s$. Since $rr^{*} = \idpart^{\otimes t(p)}$, we have $t' = sr^{*} = t$ and hence $qp = q$.

The proof is finished by translating the above operations on $r$ to $p$.
\end{proof}

Later on, we will need to know precisely which projective partitions are dominated by a tensor product $p\otimes q$ of two projective partitions. In that context, the previous result can be restated in a more elegant way. Let us first introduce some notations.

\begin{de}
 Let $\CC$ be a category of partitions and let $p\in\Proj(a)$, $q\in\Proj(b)$ be two projective partitions. By $X_{\mathcal C}(p,q)$ we denote the set of all partitions $m\in\Proj(a+b)$ such that:
\begin{itemize}
 \item $m\preceq p\otimes q$
 \item $m\not\preceq l\otimes q$ for all $l\preceq p$ with $l\neq p$ and $l\in\CC$
 \item $m\not\preceq p\otimes r$ for all $r\preceq q$ with $r\neq q$ and $r\in\CC$
\end{itemize}
\end{de}

Recall that we label the upper points of a projective partition $p\in \ProjP(k)$ by numbers $1\leqslant a\leqslant k$ from left to right, whereas the lower points are labelled by $1'\leqslant a'\leqslant k'$.

\begin{de}\label{DeMixPart}
Let $k, l\in \N_{0}$. A \emph{$(k, l)$-mixing partition} is a projective partition $h\in \ProjP(k+l)$ such that:
\begin{itemize}
\item all blocks of $h$ have size $2$ or $4$.
\item blocks of size $2$ are either of the form $(a, a')$ or $(a, b)$ with $a\leqslant k$ and $b > k$, or likewise $(a',b')$ with $a'\leqslant k'$ and $b' > k'$.
\item blocks of size $4$ are of the form $(a, a', b, b')$ with $a\leqslant k$ and $b > k$.
\end{itemize}
\end{de}

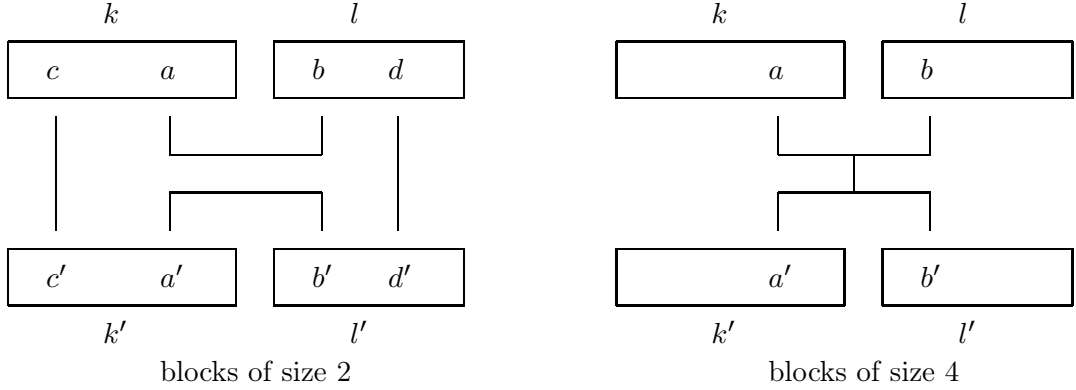
\setlength{\unitlength}{0.5cm}
\newsavebox{\framex}
\savebox{\framex}
{\begin{picture}(6,2)
\put(0,0){\line(1,0){6}}
\put(0,0){\line(0,1){1.5}}
\put(0,1.5){\line(1,0){6}}
\put(6,0){\line(0,1){1.5}}
\end{picture}}
\newsavebox{\framey}
\savebox{\framey}
{\begin{picture}(5,2)
\put(0,0){\line(1,0){5}}
\put(0,0){\line(0,1){1.5}}
\put(0,1.5){\line(1,0){5}}
\put(5,0){\line(0,1){1.5}}
\end{picture}}

\begin{center}
\begin{picture}(28,12)
 \put(4,1){blocks of size 2}
 \put(0,3){\usebox{\framex}}
 \put(1,3.5){$c'$}
 \put(4,3.5){$a'$}
 \put(2.5,2){$k'$}
 \put(7,3){\usebox{\framey}}
 \put(8,3.5){$b'$}
 \put(10,3.5){$d'$}
 \put(9,2){$l'$}
 \put(0,5){\upparti{3}{1}}
 \put(0,5){\uppartii{1}{4}{8}}
 \put(0,9){\partii{1}{4}{8}}
 \put(0,5){\upparti{3}{10}}
 \put(0,8.5){\usebox{\framex}}
 \put(1,9){$c$}
 \put(4,9){$a$}
 \put(2.5,10.5){$k$}
 \put(7,8.5){\usebox{\framey}}
 \put(8,9){$b$}
 \put(10,9){$d$}
 \put(9,10.5){$l$}
 \put(20,1){blocks of size 4}
 \put(16,3){\usebox{\framex}}
 \put(20,3.5){$a'$}
 \put(18.5,2){$k'$}
 \put(23,3){\usebox{\framey}}
 \put(24,3.5){$b'$}
 \put(25,2){$l'$}
 \put(16,6){\upparti{1}{6}}
 \put(16,5){\uppartii{1}{4}{8}}
 \put(16,9){\partii{1}{4}{8}}
 \put(16,8.5){\usebox{\framex}}
 \put(20,9){$a$}
 \put(18.5,10.5){$k$}
 \put(23,8.5){\usebox{\framey}}
 \put(24,9){$b$}
 \put(25,10.5){$l$}
\end{picture}
\end{center}

\begin{de}\label{DeMixPartAndPQ}
If $p$ and $q$ are projective partitions and if $h$ is a $(t(p), t(q))$-mixing partition, we set
\begin{equation*}
p\ast_{h}q = (p_{u}^{*}\otimes q_{u}^{*})h(p_{u}\otimes q_{u})
\end{equation*}
\end{de}

Intuitively, the operation $\ast_{h}$ mixes the through-block structures of $p$ and $q$, hence the name. Before describing the projective partitions dominated by a tensor product, let us give some elementary properties of the mixing operation.

\begin{lem}\label{LemTensorSubproj}
Let $l, l'\in \ProjP(a)$ and $r, r'\in\ProjP(b)$ be projective partitions. Then, for every $(t(l), t(r))$-mixing partition $h$ and every $(t(l'),t(r'))$-mixing partition $h'$, the following holds:
\begin{enumerate}
\item Let $\CC$ be a category of partitions. If $l \ast_{h} r\in\CC$, then $l, r\in\CC$.
\item If $l\ast_{h}r = l'\ast_{h'}r'$, then $l = l'$, $r = r'$ and $h = h'$.
\item $l \ast_{h} r\preceq l\otimes r$.
\end{enumerate} 
\end{lem}

\begin{proof}
Let $l = l_{u}^{*}l_{u}$ and $r = r_{u}^{*}r_{u}$ be the through-block decompositions. To prove $(1)$, consider the following composition:

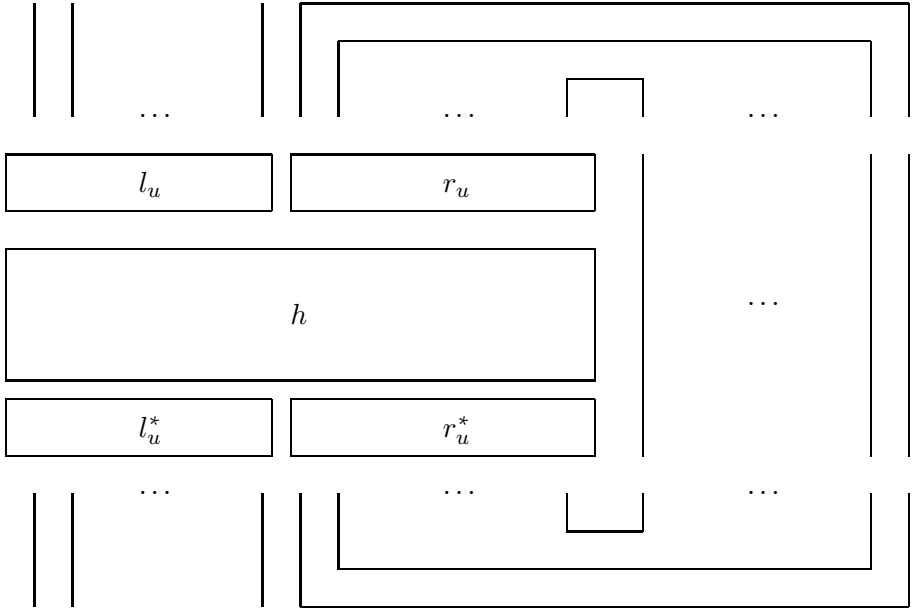
\setlength{\unitlength}{0.5cm}
\newsavebox{\framea}
\savebox{\framea}
{\begin{picture}(7,2)
\put(0,0){\line(1,0){7}}
\put(0,0){\line(0,1){1.5}}
\put(0,1.5){\line(1,0){7}}
\put(7,0){\line(0,1){1.5}}
\end{picture}}
\newsavebox{\frameb}
\savebox{\frameb}
{\begin{picture}(8,2)
\put(0,0){\line(1,0){8}}
\put(0,0){\line(0,1){1.5}}
\put(0,1.5){\line(1,0){8}}
\put(8,0){\line(0,1){1.5}}
\end{picture}}
\newsavebox{\framec}
\savebox{\framec}
{\begin{picture}(16,4)
\put(0,0){\line(1,0){15.5}}
\put(0,0){\line(0,1){3.5}}
\put(0,3.5){\line(1,0){15.5}}
\put(15.5,0){\line(0,1){3.5}}
\end{picture}}

\begin{center}
\begin{picture}(30,18)
 \put(3,13){\upparti{3}{1}}
 \put(3,13){\upparti{3}{2}}
 \put(7,13){$\ldots$}
 \put(3,13){\upparti{3}{7}}
 \put(3,13){\uppartii{3}{8}{24}}
 \put(3,13){\uppartii{2}{9}{23}}
 \put(3,13){\uppartii{1}{15}{17}}
 \put(15,13){$\ldots$}
 \put(23,13){$\ldots$}
 \put(3,4){\parti{3}{1}}
 \put(3,4){\parti{3}{2}}
 \put(7,3){$\ldots$}
 \put(3,4){\parti{3}{7}}
 \put(3,4){\partii{3}{8}{24}}
 \put(3,4){\partii{2}{9}{23}}
 \put(3,4){\partii{1}{15}{17}}
 \put(15,3){$\ldots$}
 \put(23,3){$\ldots$}
 \put(19,13){\parti{8}{1}}
 \put(23,8){$\ldots$}
 \put(19,13){\parti{8}{7}}
 \put(19,13){\parti{8}{8}}
 \put(3.5,10.5){\usebox{\framea}} 
 \put(7,11){$l_u$}
 \put(11,10.5){\usebox{\frameb}}
 \put(15,11){$r_u$}
 \put(3.5,4){\usebox{\framea}} 
 \put(7,4.5){$l_u^*$}
 \put(11,4){\usebox{\frameb}}
 \put(15,4.5){$r_u^*$}
 \put(3.5,6){\usebox{\framec}}
 \put(11,7.5){$h$}
\end{picture}
\end{center}
The pair partitions outside of $l*_h r$ effect a rotation of $r_{u}$ below $r_{u}^{*}$, so that it in fact composes to $r_{u}r_{u}^{*} = \idpart^{\otimes b}$. Hence, if a point $a\leqslant k$ is connected by $h$ to a point $b > k$, it is also connected to $b' > k'$, since $r_{u}r_{u}^{*} = \idpart^{\otimes b}$. As $h$ is symmetric, the points $a'$ and $b'$ are also connected by $h$, thus we obtain a pair $(a, a')$ between $l_{u}$ and $l_{u}^{*}$.  Therefore, the whole procedure yields the partition $l$ in the end, and $l\in\CC$ whenever $l \ast_{h} r\in \CC$ (note that $\idpart, \baarpart, \paarpart\in\CC$). Likewise, $r\in\CC$.

The above computation also proves that $l = l'$ and $r = r'$. Since
\begin{equation*}
(l_{u}\otimes r_{u})(l_{u} \ast_{h} r_{u})(l_{u}^{*}\otimes r_{u}^{*}) = h,
\end{equation*}
we also deduce $h = h'$. Hence, $(2)$ is proved.

To prove $(3)$, we compute
\begin{equation*}
(l \ast_{h} r)(l\otimes r) = (l_{u}^{*}\otimes r_{u}^{*}) h (l_{u}\otimes r_{u})(l\otimes r)
\end{equation*}
Using $(l\otimes r) = (l_{u}^{*}\otimes r_{u}^{*}) (l_{u}\otimes r_{u})$ and $ (l_{u}\otimes r_{u})(l_{u}^{*}\otimes r_{u}^*) = \idpart^{\otimes (a + b)}$, we get
\begin{equation*}
(l \ast_{h} r)(l\otimes r) = l \ast_{h} r.
\end{equation*}
\end{proof}

\begin{prop}\label{PropDominatedTensor}
Let $\CC$ be a category of partitions, let $p\in \Proj(a)$ and $q\in \Proj(b)$ and let $m\in X_{\CC}(p, q)$. Then, $m = p\ast_{h} q$ for some $(t(p), t(q))$-mixing partition $h$.
\end{prop}

\begin{proof}
We decompose $m$ as $m = (l_{u}^{*}\otimes r_{u}^{*})h(l_{u}\otimes r_{u})$ where $l_{u}\in \Pbp(a, \alpha)$ and $r_{u}\in \Pbp(b, \beta)$ are building partitions. We obtain $l_{u}$ in the following way. First, restrict $m$ to its first $a$ upper points. Denote this partition by $\widehat{l}_{u}\in P(a, 0)$. Now, for each block in $\widehat{l}_{u}$, insert a lower point and connect it to the block, if the corresponding block in $m$ is a through-block or if it is connected to some of the $b$ right upper points of $m$ (if both is the case, we only insert one point, not two). We can order the lower points of $l_{u}$ such that it is a building partition. Likewise we obtain $r_{u}$. The partition $h\in P(\alpha + \beta, \alpha + \beta)$ is now uniquely determined by the equation $m = (l_{u}^{*}\otimes r_{u}^{*})h(l_{u}\otimes r_{u})$. 

We will now show that $h$ is a $(t(l), t(r))$-mixing partition. Using $l_{u}l_{u}^{*} = \idpart^{\otimes \alpha}$ and $r_{u}r_{u}^{*} = \idpart^{\otimes \beta}$, we infer that $h = (l_{u}\otimes r_{u})m(l_{u}^{*}\otimes r_{u}^{*})$. Thus, $h$ is projective because $m$ is. 
The partition $h$ cannot connect two points $a_{1}, a_{2}\leqslant t(l)$ by definition of $l_{u}$. Otherwise, let $V_{a_{1}}$ and $V_{a_{2}}$ be the blocks in $l_{u}$ connected to $a_{1}$ (resp. $a_{2}$). Their restrictions $\widehat{V}_{a_{1}}$ and $\widehat{V}_{a_{2}}$ to upper points of $l_{u}$ would hence be connected in $m$ and thus also in $\widehat{l}_{u}$. But since $l_{u}$ is a building partition, there is at most one lower point in $l_{u}$ connected to the block containing $V_{a_{1}}$ and $V_{a_{2}}$. We infer that $h$ cannot connect two points $a_{1}, a_{2}\leqslant t(l)$ and likewise for points $b_{1}, b_{2} > t(l)$ (and their primed versions on the lower line of $h$). Thus, every block of $h$ has size at most four. Furthermore, by construction of $l_{u}$ and $r_{u}$, $h$ cannot contain singletons as blocks. Indeed, lower points appear in $l_{u}$ only if the corresponding blocks in $l_{u}$ are connected to some other blocks in $r_{u}$ or $l_{u}^{*}$. If $h$ cuts these lines, they are also cut in $m$, which is a contradiction to the construction of $l_{u}$ and $r_{u}$.

Finally, if $h$ connects a point $a\leq t(l)$ to a point $b'$ on the lower line, then also $a'$ and $b$ are connected, because $h$ is symmetric. Since $h$ is also idempotent, we infer that $a$ and $a'$ are connected, thus we obtain the block $(a, a', b, b')$. Therefore, if a block of $h$ has size strictly greater than two, it must be of size four. By symmetry, blocks of size two are either of the form $(a, a')$, $(b, b')$, $(a, b)$ or $(a', b')$ with $a\leqslant t(l)$ and $b > t(l)$. We conclude that $h$ is a $(t(l), t(r))$-mixing partition.

Setting $l = l_{u}^{*}l_{u}\in \ProjP(a)$ and $r = r_{u}^{*}r_{u}\in \ProjP(b)$, we infer that $m$ is of the form $m = l\ast_{h} r$. From Lemma \ref{LemTensorSubproj}, we deduce $l, r\in\CC$. Furthermore, $l\preceq p$ and $r\preceq q$.  
To prove this, note that $l_up$ is a part of the composition procedure in $m(p\otimes q)$ which yields $m$. Thus, if two upper points $x$ and $y$ are connected to each other in $l_up$, they are also connected in $m$, and hence in $l_u$, too. On the other hand, if $x$ and $y$ are  connected in $l_u$, then they are either connected in $p_u$, too,  or the blocks $V_x$ and $V_y$ connected to $x$ resp. to $y$ in $p_u$ are through-blocks of $p_u$ (since $m(p\otimes q)=m$). In both cases, $x$ and $y$ are connected in $l_up$ and we conclude that $l_u$ and $l_up$ coincide when restricted to their upper points. 
Let us now prove that also their through-block structure is the same. Consider a block of upper points in $l_up$. If it is not connected to some lower points of $l_up$ -- in other words, if it is a non-through-block in $l_up$ -- it yields a non-through-block in $m$, and hence in $l_u$.
Conversely, assume that $V$ is a (upper) non-through-block in $l_u$ and let $\{x_1,\ldots,x_n\}$ be the corresponding upper points in $p$, and $\{x'_1,\ldots,x'_n\}$ the corresponding lower points. Then, if a point $x_i$ gets connected by $p$ to a lower point $z'$, then $z'=x'_j$ for some $j$ (since $p$ is projective). But no point $x'_j$ is connected by $l_u$ to a lower point, since $V$ is a non-through-block of $l_u$. Thus, $V$ is also a non-through-block in $l_up$.
%
Thus, $l_{u} = l_{u}p$, which yields $l\preceq p$.

We conclude the proof by noticing that $l\neq p$ or $r\neq q$ would contradict $m\in X_{\CC}(p, q)$, since $m\preceq l\otimes r$ by Lemma \ref{LemTensorSubproj}.
\end{proof}

We can deduce a few useful lemmata on the order structure of $\Proj(k)$ from what has been done so far.

\begin{lem}
Let $\CC$ be a category of partitions and let $p\in\Proj(k)$ be a projective partition. If there is a partition $q\in\CC$ such that $t(q) < t(p)$, then there exists a \emph{projective} partition  $q'\in\CC$ with $t(q') < t(p)$ and $q'\preceq p$.
\end{lem}

\begin{proof}
Set $q'= pq^{*}qp\in \CC(k,k)$. Then, $q'$ is projective by Proposition \ref{PropProjectiveQstarQ} and $q'\preceq p$. Furthermore, using Lemma \ref{LemPPStarProj}, we have
\begin{equation*}
t(q') = t((pq^{*})(pq^{*})^{*}) = t(pq^{*})\leq \min(t(p), t(q^{*}))\leq t(q^{*}) = t(q) < t(p)
\end{equation*}
\end{proof}

\begin{lem}\label{LemComparisonProjective}
Let $p, q\in \ProjP(k)$ be two projective partitions such that $q\preceq p$. Then $t(q)\leqslant t(p)$ with equality if and only if $p = q$.
\end{lem}

\begin{proof}
By Remark \ref{RemCatOpAndBlockNumber}, we have $t(q) = t(pq)\leqslant \min(t(p), t(q))\leqslant t(p)$. If $t(p) = t(q)$, then the projections $T_{p}$ and $T_{q}$ have the same rank, by Proposition \ref{PropRankTp}. Using Lemma \ref{LemProjAndTP}, we infer $T_{p} = T_{q}$, and hence $p = q$.
\end{proof}

\begin{lem}\label{LemComparisonProjective2}
Let $p, q\in P(k, k)$ be two partitions such that $p$ is projective, $pq = q = qp$ and $t(q) = t(p)$. Then, there is a through-partition $r\in P_{2}(t(p), t(p))$ such that $q = p_{u}^{*}rp_{u}$, where $p = p_{u}^{*}p_{u}$ is the through-block decomposition of $p$.
\end{lem}

\begin{proof}
The partition $q^{*}q$ is projective and satisfies $q^{*}q\preceq p$. Furthermore, $t(q^{*}q) = t(q)$ by Lemma \ref{LemPPStarProj}. Thus, $p = q^{*}q$ by Lemma \ref{LemComparisonProjective}. Analogously we have $p = qq^{*}$. Now, if $q = q_{l}^{*}rq_{u}$ is the through-block decomposition from Proposition \ref{PropThroughBlockDecomp}, we have
\begin{equation*}
q_{u}^{*}q_{u} = q^{*}q = qq^{*} = q_{l}^{*}q_{l}
\end{equation*}
by Lemma \ref{LemPPStarProj}. By the uniqueness of the through-block decomposition of $q^{*}q$, we infer $q_{l} = q_{u}$. Furthermore, $q_{u}^{*}q_{u} = q^{*}q = p = p_{u}^{*}p_{u}$. Again by uniqueness of the through-block decomposition, we see that $q_{u} = p_{u}$.
\end{proof}

\section{Easy quantum groups}\label{SecQuantumGroups}

This short section is devoted to some preliminaries. It contains definitions and basic facts in the theory of compact quantum groups as well as an introduction to easy quantum groups.

\subsection{Compact quantum groups}

We briefly recall some definitions and results of the theory of compact quantum groups in order to fix notations. The reader is referred to \cite{woronowicz1995compact} or \cite{maes1998notes} for details and proofs.

\begin{de}\label{DefCMQG}
A \emph{compact quantum group} is a pair $\G = (C(\G), \D)$ where $C(\G)$ is a unital C*-algebra and $\D: C(\G)\rightarrow C(\G)\otimes C(\G)$ is a unital $*$-homomorphism such that $(\D\otimes \ii)\circ\D = (\ii \otimes \D)\circ \D$ and the linear spans of $\D(C(\G))(1\otimes C(\G))$ and $\D(C(\G))(C(\G)\otimes 1)$ are dense in $C(\G)\otimes C(\G)$. 
\end{de}

Here, $\ii$ denotes the identity map of the C*-algebra $C(\G)$ and $\D$ is called the \emph{coproduct}.

\begin{de}
Let $\G$ be a compact quantum group. A \emph{representation} of $\G$ of dimension $n$ is a matrix $(u_{i, j})\in M_{n}(C(\G)) \simeq C(\G) \otimes M_{n}(\C)$ such that $\D(u_{i, j}) = \sum_{k} u_{i, k}\otimes u_{k, j}$ for every $1 \leqslant i, j \leqslant n$.
\end{de}

All the compact quantum groups considered in this paper will be \emph{compact matrix quantum groups}. This means that we are given a distinguished representation $u$ of $\G$, the coefficients of which generate a dense subalgebra of $C(\G)$. This gives a more algebraic description of $\G$: Consider a Hopf $*$-algebra $\mathcal{A}$ generated by $N^{2}$ elements $(u_{i, j})_{1\leqslant i, j \leqslant N}$, where the structure maps are given by:
\begin{eqnarray*}
\D(u_{i, j}) & = & \sum_{k} u_{i, k}\otimes u_{k, j} \\
\varepsilon(u_{i, j}) & = & \delta_{i, j} \\
S(u_{i, j}) & = & u_{j, i}^{*}
\end{eqnarray*}
Furthermore, the matrices $u=(u_{i,j})$ and $u^t$ are invertible.
Then, $\D$ extends to a bounded map on the envelopping C*-algebra $C(\G)$ of $\mathcal{A}$, yielding a compact quantum group $(C(\G), \Delta)$.

An \emph{intertwiner} between two representations $u$ and $v$ of dimension respectively $n$ and $m$ is a linear map $T: \C^{n} \rightarrow \C^{m}$ such that
\begin{equation*}
(\ii \otimes T)u = v(\ii \otimes T).
\end{equation*}
The set of intertwiners between $u$ and $v$ is denoted $\Hom(u, v)$. If there exists a unitary intertwiner between $u$ and $v$, they are said to be \emph{unitarily equivalent}. A representation is said to be \emph{irreducible} if its only self-intertwiners are the scalar multiples of the identity. The \emph{tensor product} of the two representations $u$ and $v$ is the representation
\begin{equation*}
u\otimes v = u_{12}v_{13}\in C(\G)\otimes M_{n}(\C)\otimes M_{m}(\C) \simeq C(\G)\otimes M_{nm}(\C).
\end{equation*}

\begin{rem}
Here we used the \emph{leg-numbering} notations: For an operator $X$ acting on a tensor product, we set $X_{12} = X\otimes1$, $X_{23} = 1\otimes X$ and $X_{13} = (\Sigma\otimes 1)(1\otimes X)(\Sigma\otimes 1)$.
\end{rem}

The advantage of the notion of compact quantum group is that the classical Peter-Weyl theory extends to this setting.

\begin{thm}[Woronowicz]
Every unitary representation of a compact quantum group is unitarily equivalent to a direct sum of irreducible unitary representations. Moreover, any irreducible representation is finite-dimensional.
\end{thm}

To describe the representation theory of a compact quantum group, one needs to be able to split any tensor product of irreducible representations into a sum of irreducible representations. The explicit formul\ae{} for such splittings are called \emph{fusion rules}. From the point of view of Pontryagin duality, the fusion rules can be thought of as the "group law" of the dual discrete quantum group.

Compact quantum group also satisfy a version of Tannaka-Krein duality, proved by S.L. Woronowicz in \cite{woronowicz1988tannaka}. Let us give a non-standard version of this result, which is fundamental in the definition of easy quantum groups. For a compact matrix quantum group $(\G, u)$, let us denote by $\Hom_{\G}(k, l)$ the space of intertwiners between $u^{\otimes k}$ and $u^{\otimes l}$. The "Tannakanian philosophy" can be expressed in the following way :

\begin{philo}
There is a one-to-one correspondence between compact matrix quantum groups $(\G, u)$ such that $u\simeq \overline{u}$ and families $\{\Hom_{\G}(k, l)\}_{k, l\in \N_{0}}$ of intertwiner spaces.
\end{philo}

\subsection{Easy quantum groups}

In 1995 and 1998, S. Wang introduced (\cite{wang1995free} and \cite{wang1998quantum})  two very important examples of compact matrix quantum groups. For $N\in \N$, the universal C*-algebras
\begin{eqnarray*}
A_{o}(N) & = & C^{*}(u_{i, j}, 1\leqslant i, j\leqslant N | u_{i, j}^{*} = u_{i, j}, \\
& & \sum_{k} u_{i, k}u_{j, k} = \sum_{k = 1}^{N} u_{k, i}u_{k, j} = \delta_{i, j}) \\
A_{s}(N) & = & C^{*}(u_{i, j}, 1\leqslant i, j\leqslant N | u_{i, j}^{*} = u_{i, j} = u_{i, j}^{2}, u_{i, k}u_{j, k} = u_{k, i}u_{k, j} = 0, i\neq j, \\
& & \sum_{k = 1}^{N} u_{i, k} = \sum_{k = 1}^{N} u_{k, j} = 1)
\end{eqnarray*}
can be endowed with the comultiplications $\Delta(u_{i,j})=\sum u_{i,k}\otimes u_{k,j}$ turning them into compact quantum groups in the sense of Definition \ref{DefCMQG}. They are denoted $O_{N}^{+}$ and $S_{N}^{+}$ respectively. Note that if we add the relations $u_{i, j}u_{k, l} = u_{k, l}u_{i, j}$ to $A_{o}(N)$ (resp. $A_{s}(N)$), we obtain the C*-algebra $C(O_{N})$ (resp. $C(S_{N})$) of continous functions on the orthogonal group $O_{N}$ (resp. on the symmetric group $S_{N}$). Hence, $O_{N}^{+}$ is called the \emph{free orthogonal quantum group}, whereas $S_{N}^{+}$ is called the \emph{free symmetric quantum group}. 

Using the linear maps $T_{p}$ from Section \ref{SecPartitions} indexed by partitions $p\in P$, a basis of the intertwiner spaces of $O_{N}^{+}$ and $S_{N}^{}+$ can be given very explicitely:
\begin{eqnarray*}
\Hom_{O_{N}^{+}}(k, l) & = \Span\{T_{p} | p\in NC_{2}(k, l)\} \\
\Hom_{S_{N}^{+}}(k, l) & = \Span\{T_{p} | p\in NC(k, l)\}
\end{eqnarray*}
Likewise, the groups $O_{N}$ and $S_{N}$ can be seen as quantum groups, i.e. we can equip the function algebras $C(O_{N})$ and $C(S_{N})$ with comultiplications such that they are compact matrix quantum groups in the sense of Section 3.1. Their intertwiner spaces are given by:
\begin{eqnarray*}
\Hom_{O_{N}}(k, l) & = \Span\{T_{p} | p\in P_{2}(k, l)\} \\
\Hom_{S_{N}}(k, l) & = \Span\{T_{p} | p\in P(k, l)\}
\end{eqnarray*}
By functoriality of the intertwiner spaces, we know that for any compact matrix quantum group $S_{N}\subset \G \subset O_{N}^{+}$, its intertwiner space fulfills:
\begin{equation*}
\Span\{T_{p} | p\in P(k, l)\}\supset \Hom_{\G}(k, l)\supset  \Span\{T_{p} | p\in NC_{2}(k, l)\}
\end{equation*}

In 2009, T. Banica and R. Speicher came up with the following natural definition \cite[Def 6.1]{banica2009liberation}.

\begin{de}
A compact matrix quantum group $S_{N}\subset \G\subset O_{N}^{+}$ is called \emph{easy} if its intertwiner spaces are of the form
\begin{eqnarray*}
\Hom_{\G}(k, l) = \Span\{T_{p} | p\in \CC(k, l)\},\qquad k, l\in\N_{0}.
\end{eqnarray*}
The collection of sets $\CC(k, l)\subset P(k, l)$ is denoted by $\CC$.
\end{de}

Since intertwiner spaces of compact quantum groups form tensor categories, the above collection of sets $\CC\subset P$ is a category of partitions in the sense of Definition \ref{DeCatPart}. Examples of easy quantum groups include $S_{N}, O_{N}, S_{N}^{+}$ and $O_{N}^{+}$. By the "tannakanian philosophy", categories of partitions and intermediate easy quantum groups in between $S_{N}$ and $O_{N}^{+}$ are in one-to-one correspondence. Hence, easy quantum groups carry a lot of combinatorial data. In this sense, they could also be called \emph{partition quantum groups}. We refer to \cite{banica2009liberation}, \cite{banica2010classification} and \cite{weber2012classification} for details concerning easy quantum groups.

\section{From projective partitions to representations}\label{SecRepresentation}

We now start our study of the representation theory of a general easy quantum group $\G$ from the point of view of its category of partitions $\CC$. This section consists in a comprehensive study of a particular family of representations of $\G$, built out of projective partitions. We make no particular assumption on the category of partitions $\CC$. In general, we will \emph{not} be able to give a full description of all irreducible representations and their fusion rules, but we will see that the coarser structure of the set of combinatorial representations we build is quite well-behaved and already yields a lot of information on the quantum group. In special cases (see Section \ref{SecSpecial}), it is enough for deducing the complete fusion rules.

\subsection{Subrepresentations associated to partitions}

The first step is to understand the decomposition of $u^{\otimes k}$ into a sum of "small" subrepresentations. Here, "small" means being a sum of a small number of irreducible representations. We know that the space of intertwiners $\Hom(u^{\otimes k}, u^{\otimes k})$ is spanned by the operators $T_{p}$ for $p\in \CC$ and that subrepresentations of $u^{\otimes k}$ exactly correspond to orthogonal projections in this space. Thus, we can first look at partitions $p$ such that the associated operator $T_{p}$ is a projection, i.e. we look at projective partitions. We first want to investigate the following natural question: How far can we split the $k$-th tensor product of the fundamental representation using these partitions? 

We begin with assigning subrepresentations of $u^{\otimes k}$ to projective partitions $p\in \Proj(k)$. A naive idea would be to take the subrepresentation $(\ii\otimes T_{p})(u^{\otimes k})$ for $p\in \Proj(k)$. However, this representation would be far from being irreducible since for example the identity partition $\idpart^{\otimes k}$ would yield the representation $u^{\otimes k}$. We can nevertheless take advantage of the order structure on projective partitions to refine these projections.

\begin{de}\label{DeSubreprAssocToPartitions}
Let $p\in \Proj(k)$ be a projective partition, write $q\prec p$ if $q\preceq p$ and $q\neq p$ and set
\begin{equation*}
R_{p} = \bigvee_{q\in \Proj(k),\; q \prec p} T_{q},
\end{equation*}
where the supremum is taken among \emph{all projections of $\Hom(u^{\otimes k}, u^{\otimes k})$}. We define a projection
\begin{equation*}
P_{p} = T_{p} - R_{p} \in \Hom(u^{\otimes k}, u^{\otimes k})
\end{equation*}
and denote by $u_{p}$ the subrepresentation $(\ii\otimes P_{p})(u^{\otimes k})$ of $u^{\otimes k}$.
\end{de}

\begin{rem}
There is an ambiguity in the previous description. In fact, the projection $P_{p}$ may collapse to $0$ if the linear maps associated to partitions are not linearly independent. More precisely, the projection $R_{p}$ can be written as a linear combination of operators $T_{r}$ for partitions $r \in \CC(k, k)$ and this linear combination can yield the whole projection $T_{p}$. In that case, we are not producing any interesting subrepresentation. This is the first appearance of a problem with which we will have to deal all along this work.
\end{rem}

\begin{rem}\label{RemRp}
Note that by a straightforward induction, $R_{p}$ is the supremum of the projections $P_{q}$, where $q$ ranges over the set of projective partitions strictly dominated by $p$.
\end{rem}

The existence of the supremum projection $R_{p}$ inside $\Hom(u^{\otimes k}, u^{\otimes k})$ is ensured by a general von Neumann algebra argument (see for instance \cite[Chap V, Prop 1.1]{takesaki2002theory}). Hence, we know that $R_{p}$ is a linear combination of maps $T_{r}$ for partitions $r\in \CC(k, k)$. Since $T_{p}$ dominates $R_{p}$ by definition, we can multiply by $T_{p}$ on the left and on the right to see that $R_{p}$ is a linear combination of maps $T_{prp}$. Otherwise said, $R_{p}$ can be written as $\sum_{i} \lambda_{r_{i}}T_{r_{i}}$ where $r_{i}^{*}r_{i}$ and $r_{i}r_{i}^{*}$ are dominated by $p$. However, later on we will need to know that the partitions $r_{i}$ can be chosen to satisfy $t(r_{i}) < t(p)$. This is not completely obvious and relies on the following linear algebraic lemma.

\begin{lem}\label{LemvNProj}
Let $M$ be a finite-dimensional von Neumann algebra, let $(P_{i})_{i}$ be a family of orthogonal projections and let
\begin{equation*}
R = \bigvee P_{i}
\end{equation*}
be their supremum. Then, there is a family of minimal \emph{non-orthogonal} projections $(Q_{k})_{k}$ in $M$ such that $R = \sum_{k} Q_{k}$ and for every $k$ there is an index $i$ such that $\Ima(Q_{k})\subset \Ima(P_{i})$.
\end{lem}

\begin{proof}
We first deal with the case $M = M_{n}(\C)$. We know that $R$ is the orthogonal projection onto the linear span of the images of the projections $P_{i}$. Thus, there is a basis $(e_{l})_{1\leqslant l\leqslant s}$ of $\Ima(R)$ such that for every $l$, there is an index $i$ such that $e_{l}\in \Ima(P_{i})$. Let $(f_{t})$ be an orthogonal basis of the orthogonal complement of $\Ima(R)$, so that $(e_{l}, f_{k})_{l, k}$ is a basis of $\C^{n}$ and let $B$ be the change-of-basis matrix from this basis to the canonical basis of $\C^{n}$. This means that
\begin{equation*}
R = B^{-1}\left(\sum_{1\leqslant k\leqslant s}E_{k}\right)B,
\end{equation*}
where the $(k, k)$-th coefficient of $E_{k}$ is $1$ and all the others are $0$. Setting $Q_{k} = B^{-1}E_{k}B$ for $1\leqslant k\leqslant s$, we get minimal projections summing up to $R$. Moreover, $\Ima(Q_{k}) = \C e_{k} \subset \Ima(P_{i})$ for some $i$.

Any finite-dimensional von Neumann algebra is isomorphic to a direct sum of matrix algebras, so let us write $M = \oplus_{t} M_{n_{t}}(\C)$. Let $P_{i}^{t}$ be the $t$-th component of $P_{i}$, which is again an orthogonal projection. Let $R^{t}$ be the supremum of the family of projections $(P_{i}^{t})$ for a fixed $t$. Then, $R$ is the supremum of the family of projections $R^{t}$. But these projections are pairwise orthogonal since the belong to different summands. Thus, $R = \oplus R^{t}$. Now, by the first part of the proof we know that each $R^{t}$ can be written as a sum of minimal projections $Q_{k}^{t}$ such that for every $k$ there is an index $i$ such that $\Ima(Q_{k}^{t}) \subset \Ima(P_{i}^{t}) \subset \Ima(P_{i})$. But now, $R$ is the sum of all the projections $Q_{k}^{t}$ and the proof is complete.
\end{proof}

\begin{prop}
Let $p$ be a projective partition and assume that $R_{p}\neq 0$. Then, $R_{p}$ can be written as a linear combination of operators $T_{r}$ with $t(r) < t(p)$.
\end{prop}

\begin{proof}
Let $R_{p} = \sum Q_{k}$ be the decomposition given by Lemma \ref{LemvNProj}. This means that for every $k$, there is a projective partition $q$ strictly dominated by $p$ such that $T_{q}Q_{k} = Q_{k}$. Now, each $Q_{k}$ can be written as a linear combination of operators $T_{r}$, and multiplying by the appropriate $T_{q}$ shows that we can replace $r$ by $qr$. Since $t(qr) \leqslant t(q) < t(p)$, the proof is complete.
\end{proof}

We end this subsection with a lemma which will be used repeatedly in the sequel.

\begin{lem}\label{LemGeneralt}
Let $p\in \Proj(k)$ and let $q\in \CC(k, k)$ be such that $t(pqp) < t(p)$. Then, $P_{p}T_{q}P_{p} = 0$.
\end{lem}

\begin{proof}
Set
\begin{equation*}
V = T_{p}T_{q}T_{p} \in \C  T_{pqp}.
\end{equation*}
Using the fact that $R_{p}T_{p} = R_{p} = T_{p}R_{p}$, we see that
\begin{equation}\label{EqSubprojection}
P_{p}T_{q}P_{p} = V - R_{p}V - VR_{p} + R_{p}VR_{p}.
\end{equation}
If $t(pqp) < t(p)$, then $T_{pqp}$ is a partial isometry associated to two strict subprojections of $T_{p}$, i.e. it is dominated by $R_{p}$ by Lemma \ref{LemComparisonProjective}. Hence, $R_{p}V = VR_{p} = V$, implying that the right-hand side of Equation (\ref{EqSubprojection}) is $0$.
\end{proof}

\subsection{Irreducibility}

The next step is to determine whether the representations $u_{p}$ are irreducible or not. This question can be answered affirmatively in the case of categories of noncrossing partitions in Section \ref{SecSpecial}, but the general case is quite complicated. We solve it by giving an indirect description of the space $\Aut(u_{p}) = \Hom(u_{p}, u_{p})$, i.e. we relate it to the irreducible representations of certain groups. To do so, we first introduce a useful tool. Recall that to any permutation $\sigma\in S_{k}$ a pair partition $r_{\sigma}\in P_{2}(k, k)$ is associated where the $i$-th point in the upper row is connected to the $\sigma(i)$-th point in the lower row.

\begin{de}\label{DefSymmetryGroup}
Let $p\in \CC$ be a projective partition, such that $t(p) > 0$, with through-block decomposition $p = p_{u}^{*}p_{u}$. For $\sigma\in S_{t(p)}$, set $p_{\sigma} = p_{u}^{*}r_{\sigma}p_{u}$. Then, its \emph{symmetry group} (relative to $\CC$) is the set
\begin{equation*}
\Sym(p) = \{\sigma\in S_{t(p)}| p_{\sigma}\in \CC(k, k)\}.
\end{equation*}
\end{de}

\begin{rem}
If $t(p) = 0$, then $u_{p}$ is one-dimensional, hence irreducible and there is nothing to do.
\end{rem}

\begin{rem}\label{RemNCTrivialSym}
The only noncrossing through-partition $r\in P_{2}(k, k)$ is $r = \idpart^{\otimes k}$. Thus, $\Sym(p)$ is trivial whenever $\CC$ is a category of noncrossing partitions.
\end{rem}

\begin{ex}
As an example, let us compute the symmetry group of $p = \idpart^{\otimes k}$ for $k \geqslant 2$ (the computation for $k = 1$ being trivial). If $\CC$ contains $\crosspart\in P(2, 2)$ (i.e. if $\G$ is a classical group), then $\Sym(p)$ contains all transpositions, hence is equal to $S_{k}$. If $\CC$ contains the half-liberating partition $\halflibpart\in P(3,3)$ but not the crossing $\crosspart$ (i.e. if $\G$ is \emph{half-liberated}), then $\Sym(p)$ contains two commuting subgroups which are the permutation groups respectively of odd and even indices. Thus, if $k = 2k'$ is even then $\Sym(p) = S_{k}\times S_{k'}$ and if $k = 2k'+1$ then $\Sym(p) = S_{k'+1}\times S_{k'}$. 
\end{ex}

As the name indicates, this set is in fact a subgroup of $S_{t(p)}$. Let us prove this.

\begin{prop}
The set $\Sym(p)$ is a subgroup of $S_{t(p)}$ for every projective partition $p$.
\end{prop}

\begin{proof}
Thanks to the equality $p_{u}p_{u}^{*} = \idpart^{\otimes t(p)}$ from Lemma \ref{LemPPStar}, we get, for $\sigma, \sigma'\in \Sym(p)$,
\begin{equation*}
p_{\sigma\sigma'} = p_{\sigma}p_{\sigma'} \in \CC
\end{equation*}
and
\begin{equation*}
p_{\sigma^{-1}} = p_{\sigma}^{*} \in \CC.
\end{equation*}
Observing that $p_{id} = p$ concludes the proof.
\end{proof}

\begin{rem}
The map $\sigma \mapsto \sigma^{-1}$ gives a bijection between $\Sym(p)$ and $\Sym(p^{*})$. Note also that the natural inclusion $\Sym(p)\times \Sym(q) \rightarrow \Sym(p\otimes q)$ is not surjective in general (consider e.g. $p = q = \idpart$).
\end{rem}

To decompose $u_{p}$ into irreducible subrepresentations, we have to find minimal projections in the space $\Aut(u_{p})$. Note that by definition, $P_{p}T_{p_{\sigma}}P_{p}\in \Aut(u_{p})$ for every $\sigma\in \Sym(p)$. Our statement is that these maps generate the whole space of self-intertwiners. Before proving it, we need two lemmata.

\begin{lem}\label{LemSymP}
Let $p\in \Proj(k)$ and let $q\in \CC(k, k)$ be such that $t(pqp) = t(p)$. Then, there is a permutation $\sigma\in \Sym(p)$ such that $pqp = p_{\sigma}$.
\end{lem}

\begin{proof}
Set $r = pqp$. Applying Lemma \ref{LemComparisonProjective2}, we get $r = p_{u}^{*}r_{m}p_{u}$, i.e. $r = p_{\sigma}$ for some $\sigma\in \Sym(p)$.
\end{proof}

\begin{lem}\label{LemMultiplicativePsi}
Let $p$ be a projective partition and let $\sigma, \sigma'\in \Sym(p)$. Then,
\begin{equation*}
(P_{p}T_{p_{\sigma}}P_{p})(P_{p}T_{p_{\sigma'}}P_{p}) =  P_{p}T_{p_{\sigma\sigma'}}P_{p}
\end{equation*}
\end{lem}

\begin{proof}
Let $V$ denote the left-hand side of the equation, we then have
\begin{eqnarray*}
V & = & P_{p}T_{p_{\sigma}}P_{p}T_{p_{\sigma'}}P_{p} \\
& = & P_{p}T_{p_{\sigma}}T_{p}T_{p_{\sigma'}}P_{p} - P_{p}T_{p_{\sigma}}R_{p}T_{p_{\sigma'}}P_{p} \\
& = & P_{p}T_{p_{\sigma}p_{\sigma'}}P_{p} - P_{p}T_{p_{\sigma}}R_{p}T_{p_{\sigma'}}P_{p}
\end{eqnarray*}
Here, we used that $\gamma(p_\pi,p_{\pi'})=\gamma(p,p)=0$ for all permutations $\pi,\pi'\in\Sym(p)$ (note that $\beta(p_\pi)=\beta(p)$ and $\rl(p_\pi,p_{\pi'})=\rl(p,p)$).
Now, $T_{p_{\sigma}}R_{p}T_{p_{\sigma'}}$ is a linear combination of operators $T_{pqp}$, where $t(pqp)\leqslant t(q) < t(p)$, hence
\begin{equation*}
P_{p}T_{p_{\sigma}}R_{p}T_{p_{\sigma'}}P_{p} = 0
\end{equation*}
by Lemma \ref{LemGeneralt} and the proof is complete.
\end{proof}

The next result will be stated in the language of group algebras: If $\Gamma$ is a finite group, its \emph{group algebra} is the vector space $\C[\Gamma]$ generated by elements $\delta_{\gamma}$ for $\gamma\in \Gamma$ with the algebra structure induced by the group law.

\begin{prop}\label{PropSymP}
For every projective partition $p$, the map
\begin{equation*}
\Psi: \left\{\begin{array}{ccc}
\C[\Sym(p)] & \rightarrow & \Aut(u_{p}) \\
\delta_{\sigma} & \mapsto & P_{p}T_{p_{\sigma}}P_{p}
\end{array}\right.
\end{equation*}
extends to a surjective $*$-homomorphism. If moreover the maps $T_{q}$ are linearly independent for $q\in \CC(t(p), t(p))$ and if $P_{p}\neq 0$, then $\Psi$ is an isomorphism.
\end{prop}

\begin{proof}
First note that $\Psi$ is multiplicative by Lemma \ref{LemMultiplicativePsi}. By definition, $\Aut(u_{p})$ is the linear span of operators $P_{p}T_{q}P_{p}$ for $q\in \CC(k, k)$. According to Lemma \ref{LemGeneralt}, such operators are zero if $t(pqp) < t(p)$. Assume that $t(pqp) = t(p)$. By Lemma \ref{LemSymP}, $pqp = p_{\sigma}$ for some $\sigma$ in $\Sym(p)$. Moreover,
\begin{equation*}
P_{p}T_{q}P_{p} = P_{p}T_{p}T_{q}T_{p}P_{p} \in \C  P_{p}T_{pqp}P_{p} = \C \Psi(\delta_{\sigma}).
\end{equation*}
This proves the surjectivity of the map $\Psi$.

To prove injectivity, let $\lambda_{i}$ be scalars and let $\sigma_{i}$ be permutations in $\Sym(p)$ such that
\begin{equation}\label{EqPropSymP}
\sum_{i} \lambda_{i}\Psi(\delta_{\sigma_{i}}) = 0.
\end{equation}
We know that
\begin{equation*}
\Psi(\delta_{\sigma_{i}}) = T_{p_{\sigma_{i}}} - R_{p}T_{p_{\sigma_{i}}} - T_{p_{\sigma_{i}}}R_{p} + R_{p}T_{p_{\sigma_{i}}}R_{p}.
\end{equation*}
Noticing that the coefficient of $T_{p_{\sigma_{i}}}$ in Equation (\ref{EqPropSymP}) is $\lambda_{i}$, we get $\lambda_{i} = 0$ for every $i$ by linear independence.
\end{proof}

We take the occasion of settling the question of linear independence. The following result is well-known but we include an argument since we could not find a proof in the literature.

\begin{lem}\label{lem:linearindependence}
Let $N$ be an integer and let $\CC$ be a category of \emph{noncrossing} partitions. Then, the maps $(T_{p})_{p\in \CC(k, l)}$ are linearly independent for every $k$ and $l$ if and only if $N\geqslant 4$.
\end{lem}

\begin{proof}
It is enough to prove it for $\CC = NC$ and, up to rotating the partitions, for the maps $(T_{p})_{p\in NC(0, k)}$ and even for the maps $(\mathring{T}_{p})_{p\in NC(0, k)}$. In that case, consider
\begin{equation*}
\mathring{T}_{p} = \C \rightarrow (\C^{N})^{\otimes k}
\end{equation*}
as a vector (identifying it with $\mathring{T}_{p}(1)$) and form the Gram matrix $M(k, N)\in M_{\vert NC(0, k)\vert}(\C)$ of this family. Since
\begin{equation*}
\langle \mathring{T}_{p}, \mathring{T}_{q}\rangle = \mathring{T}_{p}^{*}\mathring{T}_{q} = N^{\rl(p, q)},
\end{equation*}
$M(k, N)$ is exactly the \emph{matrix of chromatic joins} introduced by W.T. Tutte in \cite{tutte1993matrix}, where he also computed its determinant. According to \cite[Equation 16]{copeland2002two}, this determinant can be written as
\begin{equation*}
\det(M(k, N)) = N^{C_{k}}\prod_{i=2}^{k}\left(\frac{U_{i}(\sqrt{N})}{U_{i-2}(\sqrt{N})}\right)^{b(k, i)},
\end{equation*}
where $C_{k}$ and $b(k, i)$ are integers and $U_{i}$ is the $i$-th \emph{dilated Chebyshev polynomial of the second kind} satisfying $U_{0} = 1$, $U_{1} = X$ and $XU_{k} = U_{k+1} + U_{k-1}$. The above equation tells us that the vectors $(\mathring{T}_{p})_{p\in NC(0, k)}$ are linearly independent if and only if $N$ is not the square of a root of a Chebyshev polynomial $U_{i}$. Since roots of $U_{i}$ all have absolute value strictly smaller than $2$, the result follows (it is also easy to see that $1$, $\sqrt{2}$ and $\sqrt{3}$ are roots of $U_{2}$, $U_{3}$ and $U_{5}$ respectively).
\end{proof}

Note that in contrast to the noncrossing case, even for $P_{2}$ the maps $(T_{p})_{p\in P_{2}(0, k)}$ cannot be linearly independent for all $k$. In fact, the cardinality of $P_{2}(0, k)$ grows much faster than $k\times N$, which is the dimension of $(\C^{N})^{\otimes k}$.

Even if $\Psi$ is not an isomorphism, we can determine its kernel. Let us recall some facts about the representation theory of finite groups. If $\Gamma$ is a finite group, the group algebra $\C[\Gamma]$ admits a direct sum decomposition
\begin{equation*}
\C[\Gamma] = \bigoplus_{\alpha\in \Irr(\Gamma)} M_{\dim(\alpha)}(\C).
\end{equation*}
The map $\Psi$ being a $*$-homomorphism, its kernel is an ideal in the finite-dimensional C*-algebra $\C[\Gamma]$, hence it is equal to
\begin{equation*}
\bigoplus_{\alpha\in J(p)} M_{\dim(\alpha)}(\C)
\end{equation*}
for some subset $J(p)$ of $\Irr(\Gamma)$. Thus, if $I(p)$ denotes the (set-theoretic) complement of $J(p)$, then $\Psi$ yields an isomorphism from the sum of matrix algebras indexed by $I(p)$ onto $\Aut(u_{p})$.

\subsection{Unitary equivalence}

Before dealing with the decomposition of $u^{\otimes k}$ let us decide whether two given projective partitions yield unitarily equivalent representations. Let us first restate the problem. Let $p\in \Proj(k)$ and $q\in \Proj(l)$ be projective partitions. A morphism $W$ between the representations $u_{p}$ and $u_{q}$ can always be extended by $0$ to give an element $V\in \Hom(u^{\otimes k}, u^{\otimes l})$ satisfying
\begin{equation*}
VP_{p} = V=P_{q}V.
\end{equation*}
Reciprocally, if $V$ is such a morphism, then $W = P_{q}VP_{p}$ can be seen as an element in $\Hom(u_{p}, u_{q})$. Moreover, $W$ is unitary in $\Hom(u_{p}, u_{q})$ if and only if $V$ is a partial isometry with range and support projections respectively $P_{q}$ and $P_{p}$, i.e.
\begin{equation*}
VV^{*} = P_{q} \text{ and } V^{*}V = P_{p}.
\end{equation*}
There is an obvious way to build such a partial isometry, if the projective partitions $p$ and $q$ are equivalent in the following sense.

\begin{de}
 Let $\CC$ be a category of partitions. Two projective partitions $p\in \Proj(k)$ and $q\in \Proj(l)$ are \emph{equivalent} in $\CC$ (we write $p\sim q$), if there is a partition $r\in\CC(k,l)$ such that $r^*r=p$ and $rr^*=q$.
\end{de}

The equivalence of projective partitions translates exactly to the equivalence of the associated representations. 

\begin{thm}\label{ThmUnitaryEquivalence}
Let $\CC$ be a category of partitions, let $N$ be an integer and let $\G$ be the associated compact quantum group. Let $p\in \Proj(k)$ and $q\in \Proj(l)$. Then, the representations $u_{p}$ and $u_{q}$ are unitarily equivalent if and only if either $p\sim q$ in $\CC$ or $u_{p} = u_{q} = 0$.
\end{thm}

\begin{proof}
Assume that $p\sim q$, i.e. there is a partition $r\in \CC(l, k)$ such that $r^{*}r = p$ and $rr^{*} = q$.
First note that $T_{r^{*}}R_{q}T_{r}$ is a linear combination of maps of the form $T_{r^{*}lr}$ for $l\prec q$. Thus, $t(r^{*}lr) \leqslant t(l) < t(q) = t(p)$ and by Lemma \ref{LemGeneralt}, $P_{p}T_{r^{*}}R_{q}T_{r}P_{p} = 0$. From this, we get, setting $V = P_{q}T_{r}P_{p}$,
\begin{eqnarray*}
V^{*}V & = & P_{p}T_{r^{*}}P_{q}T_{r}P_{p} \\
& = & P_{p}T_{r^{*}}T_{q}T_{r}P_{p} - P_{p}T_{r^{*}}R_{q}T_{r}P_{p} \\
& = & P_{p}T_{r^{*}qr}P_{p} \\
& = & P_{p}T_{p}P_{p} \\
& = & P_{p}.
\end{eqnarray*}
Note that we used here the fact that $\gamma(q, r) = \gamma (r^{*}, qr) = 0$, which comes from the equality $rr^{*} = q$ and Lemma \ref{LemGamma}. By the same computations, we get $VV^{*} = P_{q}$, thus $u_p$ and $u_q$ are unitarily equivalent.

Conversely, 
let $W$ be a unitary in $\Hom(u_{p}, u_{q})$ and extend it to a partial isometry in $V\in \Hom(u^{\otimes k}, u^{\otimes l})$ satisfying $P_{q}V =V= VP_{p}$. Now, $V$ can be written as a finite sum
\begin{equation*}
V = \sum_{i} \lambda_{i}T_{r_{i}}
\end{equation*}
for some complex coefficients $\lambda_{i}$ and some partitions $r_{i}\in \CC(k, l)$. Since also $T_qV=V=VT_p$, we may assume that all $r_i$ fulfill $r_i=qr_ip$.

 As soon as $P_{p} \neq 0$, we can assume that there is an index $i$ such that $T_{r_{i}}P_{p} \neq 0$ since $VP_{p} = V$. This implies that $P_{p}T_{r_{i}^{*}r_{i}}P_{p}$ is non-zero. By Lemma \ref{LemGeneralt}, we then have $t(r_i^*r_i)=t(pr_{i}^{*}r_{i}p) = t(p)$. Thus, by Lemma 
\ref{LemComparisonProjective2}, $r_{i}^{*}r_{i} = p_u^*rp_u$ for some through-partition $r\in P_2(t(p),t(p))$ and the decomposition $p=p_u^*p_u$. But $r_{i}^{*}r_{i}$ is projective by Lemma \ref{LemPPStarProj}, hence has trivial through-partition. This means that $r_{i}^{*}r_{i} = p$. The same argument proves that we can find an index $i$ such that also $r_{i}r_{i}^{*} = q$ is fulfilled, which yields $p\sim q$.
\end{proof}

Note in particular that if $p\sim q$, then $t(p) = t(q)$. However, the invariant $t$ does not completely characterize the equivalence class of $p$ in general.

\begin{lem}\label{LemEquivAndT}
 Let $\CC$ be a category of partitions and let $p\in\Proj(k), q\in\Proj(l)$.
\begin{enumerate}
 \item If $p\sim q$, then $t(p)=t(q)$.
 \item If $t(p)=t(q)=t(pq)$, then $p\sim q$.
\end{enumerate}
\end{lem}
\begin{proof}
 The first assertion follows directly from Lemma \ref{LemPPStarProj}. For the second, consider $q':=pqp$ and note that $t(q')=t(pq)$ (again by Lemma \ref{LemPPStarProj}). From Lemma \ref{LemComparisonProjective}, we deduce that $pqp=q'=p$. Likewise we prove $qpq=q$, which yields $p\sim q$ using $r:=pq$.
\end{proof}

\begin{rem}\label{RemHN+}
In the previous lemma, none of the converse directions hold. If one considers for instance the hyperoctahedral quantum group $H_{N}^{+}$, the projective partitions
\begin{equation*}
p = \vierpartrot\idpart \text{ and } q = \idpart\vierpartrot
\end{equation*}
correspond to irreducible subrepresentations of $u^{\otimes 3}$ which are \emph{not} unitarily equivalent, labelled respectively by $01$ and $10$ using the notations of \cite[Thm 7.3]{banica2009fusion}. Indeed, any partition $r=r_l^*r_mr_u$ such that $p=r^*r$ and $q=rr^*$ fulfills $r_u=p_u$ and $r_l=q_u$ by the uniqueness of the through-block decomposition of $p$ and $q$ respectively. Hence $r=q_u^*r_mp_u$ consists either of two three blocks or contains a crossing, depending whether $r_m=\crosspart$ or $r_m=\idpart^{\otimes 2}$. Therefore $r$ is not contained in the category $\langle\vierpart\rangle$ associated to $H_N^+$.

An example for $p\sim q$ (again in $\langle\vierpart\rangle$) but $t(p)\neq t(pq)$ is given by the partitions:
\[p=\doublepairrot\vierpartrot \quad \textnormal{and} \quad q=\vierpartrot\doublepairrot\]
\end{rem}

If $p\in \Proj(k)$ and $q \in \Proj(l)$ are two projective partitions such that $t(p) = t(q)$, we can give another characterization of the equivalence relation $\sim$ for $p$ and $q$. For any $\sigma\in S_{t(p)}$, set
\begin{equation*}
r^{p}_{q}(\sigma) = q_{u}^{*}r_{\sigma}p_{u} \in P(k, l)
\end{equation*}
and define
\begin{equation*}
\Sym(p, q) = \{\sigma\in S_{t(p)}, r^{p}_{q}(\sigma)\in \CC(k, k)\}.
\end{equation*}

We can now give a criterion for unitary equivalence of the representations.

\begin{prop}\label{PropUnitaryEquivalence}
Let $\CC$ be a category of partitions and let $p\in \Proj(k)$ and $q\in \Proj(l)$. Then, $p\sim q$ if and only if either $\Sym(p, q) \neq \emptyset$.
\end{prop}

\begin{proof}
Obviously, for $\sigma\in \Sym(p, q)$ we have $r^{p}_{q}(\sigma)(r^{p}_{q}(\sigma))^{*} = q$ and $(r^{p}_{q}(\sigma))^{*}r^{p}_{q}(\sigma) = p$. Reciprocally, if $r$ is a partition satisfying $r^{*}r = p$ and $rr^{*} = q$, then $r_{u} = p_{u}$ and $r_{l} = q_{u}$ by uniqueness of the through-block decomposition, i.e. $r$ is of the form $r^{p}_{q}(\sigma)$.
\end{proof}

\subsection{Decomposition of the fundamental representation}

We are now ready for the decomposition of $u^{\otimes k}$.
Our first result is that the representations $u_p$ give a complete decomposition of $u^{\otimes k}$.

\begin{prop}\label{ThmGeneralDecomposition}
Let $\CC$ be a category of partitions, let $N$ be an integer and let $\G$ be the associated easy quantum group. 
If $v$ is a subrepresentation of $u^{\otimes k}$ containing all the representations $u_{p}$, $p\in\Proj(k)$, then it is equal to $u^{\otimes k}$. In this sense we write:
\begin{equation*}
u^{\otimes k} = \sum_{p\in \Proj(k)} u_{p}
\end{equation*}
\end{prop}

\begin{proof}
Let $P_v$ be the orthogonal projection associated to $v$, which hence dominates all projections $P_{p}$. We claim that $P_{v}$ in fact dominates $T_{p}$ for every projective partition $p\in \Proj(k)$. The proof will be done by induction on $t(p)$. Let $k_{1} < \dots < k_{n}$ be the possible values of $t$ on $\Proj(k)$.
\begin{itemize}
\item If $t(p) = k_{1}$, then any projective partition $q$ strictly dominated by $p$ satisfies $t(q) < t(p)$, which is impossible. Thus, $T_{p} = P_{p}$, which is dominated by $P_{v}$.
\item If $t(p) = k_{i}$ for $i>1$, then $R_{p}$ is the supremum of some operators which are all dominated by $P_{v}$ since they are associated to partitions $q$ with $t(q) < t(p) = k_{i}$ by Lemma \ref{LemGeneralt}. Hence $T_{p} = P_{p} + R_{p}$ is dominated by $P_{v}$.
\end{itemize}
Taking, $p$ to be the identity partition $\idpart^{\otimes k}$ (which is obviously projective), we see that $P_{v} = \Id$, i.e. $v = u^{\otimes k}$.
\end{proof}




Proposition \ref{ThmGeneralDecomposition} ensures, that we do not "miss" a subrepresentation of $u^{\otimes k}$ when considering the collection of $u_p, p\in\Proj(k)$. The drawback is that the representations $u_p$ do \emph{not} give a direct sum decomposition of $u^{\otimes k}$. The reason is that $P_p$ and $P_q$ are not orthogonal in general. For example, if $t(p) = t(q) = 0$, then $P_{p} = T_{p}$ and $P_{q} = T_{q}$ but $P_{p}P_{q} = \lambda T_{pq}$ cannot be $0$.

This can be solved by passing to equivalence classes of projective partitions. To a projective partition $p\in\Proj(k)$, we assign
\[P_{[p]} := \bigvee_{q\in \Proj(k),\; q \sim p} P_q\]
and we define the subrepresentation $u_{[p]}:=(\ii\otimes P_{[p]})(u^{\otimes k})$ of $u^{\otimes k}$.

\begin{prop}\label{ThmGeneralDecompositionEquivClasses}
Let $\CC$ be a category of partitions, let $N$ be an integer and let $\G$ be the associated easy quantum group. Then, we have the decomposition:
\begin{equation*}
u^{\otimes k} = \bigoplus_{[p]\in \Proj(k)/\sim} u_{[p]}
\end{equation*}
\end{prop}
\begin{proof}
 Let $p,q\in\Proj(k)$ be non-equivalent projective partitions.  By Lemma \ref{LemEquivAndT}, we know that either $t(p)\neq t(q)$ or $t(p)\neq t(pq)$. In both cases, we obtain that $P_p$ and $P_q$ are orthogonal. Indeed, assume first $t(pq)< t(p)=t(q)$. By Lemma \ref{LemGeneralt}, we infer that $P_pT_qP_p=0$ and likewise $P_pT_rP_p=0$ for all projective partitions $r\prec q$. Hence $P_pR_qP_p=0$ and thus $P_pP_qP_p=0$. Similarly, if $t(q)<t(p)$, then also $t(pq)<t(p)$ and thus $P_pP_qP_p=0$. We conclude that $P_{[p]}$ and $P_{[q]}$ are orthogonal.
 We finish the proof using Proposition \ref{ThmGeneralDecomposition}. 
\end{proof}

This result has yet another drawback: The subrepresentations $u_{[p]}$ might be much larger than the representations $u_p$, and hence even further away from being irreducible. However, a combination of Proposition \ref{ThmGeneralDecomposition} together with the symmetry groups of Definition \ref{DefSymmetryGroup} finally yields a decomposition of $u^{\otimes k}$ into irreducible subrepresentations related to partitions.



\begin{thm}\label{PropGeneralDecompositionRefined}
Let $\CC$ be a category of partitions, let $N$ be an integer and let $\G$ be the associated easy quantum group. Then, the irreducible representations of $\G$ can be labelled as $u_p(\alpha)$ where $p$ is a projective partition and $\alpha$ is an irreducible representation of $\Sym(p)$.
Furthermore, there are integers $0 \leq \nu_{p}(\alpha) \leq \dim(\alpha)$ for every pair $(p,\alpha)$ such that:
\begin{equation*}
u^{\otimes k} = \bigoplus_{p\in \Proj_{k}(\CC)}\bigoplus_{\alpha\in \Irr(\Sym(p))}\nu_{p}(\alpha)u_{p}(\alpha)
\end{equation*}
\end{thm}

\begin{proof}
Starting from Proposition \ref{ThmGeneralDecomposition}, we know that $\Id \in \Hom(u^{\otimes k}, u^{\otimes k})$ is the supremum of the projections $P_{p}$. These in turn may be written as a direct sum of projections $P_p^\gamma(\alpha)$ for $\alpha\in I(p)\subset \Irr(\Sym(p))$ and $0\leq \gamma\leq \dim(\alpha)$, by Proposition \ref{PropSymP}. By Lemma \ref{LemvNProj}, we see that $\Id$ can therefore be decomposed as a sum of minimal non-orthogonal projections, each dominated by -- and hence equal to -- some $P_p^\gamma(\alpha)$. Each of these projections gives rise to a copy of $u_p(\alpha)$.
\end{proof}

\begin{rem}
If $p\sim q$ and $\sigma\in \Sym(p, q)$, then conjugating by $\sigma$ yields an isomorphism between $\Sym(p)$ and $\Sym(q)$. If the map $\Psi$ of Proposition \ref{PropSymP} is injective, two representations $u_{p}(\alpha)$ and $u_{q}(\beta)$ are then unitarily equivalent if and only if $\alpha = \beta$. However, this fails in the general case since the images of nonequivalent minimal projections in $\C[\Sym(p)]$ may collapse to $0$ in $\Aut(u_{p})$. This makes the problem of equivalence of irreducible representations much more complicated, since it depends on the kernel of the map $p\mapsto T_{p}$.
\end{rem}




\subsection{Fusion rules}

The last part of this section is devoted to the study of some "partial fusion rules". More precisely, we adress the following question: Given two projective partitions $p$ and $q$, is it possible to write explicitely $u_{p}\otimes u_{q}$ as a direct sum of representations associated to projective partitions?

Let $p\in \Proj(k)$ and $q\in \Proj(l)$ be projective partitions and consider the projection $P_{p}\otimes P_{q}\in \Hom(u^{\otimes (k+l)}, u^{\otimes (k+l)})$. This is precisely the projection onto $u_{p}\otimes u_{q}$. From the definition, we see that it decomposes as
\begin{equation*}
(T_{p} - R_{p})\otimes (T_{q} - R_{q}) = T_{p}\otimes T_{q} - (T_{p}\otimes R_{q} + R_{p}\otimes T_{q} - R_{p}\otimes R_{q})
\end{equation*}

We are going to express this using the partitions $P_{m}$. To do so, first recal that $R_{p}$ is the supremum of the projections $P_{m}$ for $m\prec p$ and that $T_{p}$ is the supremum of the projections $P_{m}$ for $m\preceq p$.

\begin{lem}\label{LemFusion}
The operator $T_{p}\otimes R_{q} + R_{p}\otimes T_{q} - R_{p}\otimes R_{q}$ is the supremum of the projections $P_{m}$ for all projective partitions $m$ such that there exists $l\prec p$ satisfying $m\preceq l\otimes q$ or there exists $r\prec q$ satisfying $m\preceq p\otimes r$.
\end{lem}

\begin{proof}
Set $A = T_{p}\otimes R_{q}$, $B = R_{p}\otimes T_{q}$ and $C = R_{p}\otimes R_{q}$. Then, $AB = BA = C$, i.e. $C$ is the minimum of the commuting projections $A$ and $B$. This implies that $A+B-C$ is the supremum of the projections $A$ and $B$. Now, $T_{p}\otimes R_{q}$ is the sumpremum of the projections $T_{p\otimes r}$ for $r\prec q$ and $R_{p}\otimes T_{q}$ is the supremum of the projections $T_{l\otimes q}$ for $l\prec p$, hence the result.
\end{proof}

Recall that $X_{\CC}(p, q)$ denotes the set of partitions dominated by $p\otimes q$ which are not dominated by $l\otimes q$ for $l \prec p$ or by $p\otimes r$ for $r \prec q$. Lemma \ref{LemFusion} implies that $(P_{p}\otimes P_{q})P_{m}\neq 0$ if and only if $m\in X_{\CC}(p, q)$. However, the latter operator need not be a projection in general, so that we cannot directly deduce the decomposition of the tensor product from this. If $P_{m}$ is minimal (e.g. when we are considering noncrossing partitions), we can still work out the decomposition. Along the way, we will also give an explicit description of the set $X_{\CC}(p, q)$ using the mixing partitions introduced in Definition \ref{DeMixPart} and the set $Y(p, q)$ of partitions $p\ast_{h}q$ for all $(t(p), t(q))$-mixing partitions $h$.                                                                                                                                                                            

\begin{thm}\label{ThmFusionRules}
Let $\CC$ be a category of partitions. Then, $X_{\CC}(p, q) = Y(p, q) \cap \CC$. Moreover, if $P_{m}$ is minimal for all $m\in X_{\CC}(p, q)$, then
\begin{equation*}
u_{p}\otimes u_{q} = \sum_{m\in X_{\CC}(p, q)}v_{m},
\end{equation*}
where $v_{m}$ is equivalent to $u_{m}$ and the sum means that the smallest representation of $\G$ containing $v_{m}$ for every $m\in X_{\CC}(p, q)$ is $u_{p}\otimes u_{q}$.
\end{thm}

\begin{proof}
In view of Proposition \ref{PropDominatedTensor}, for the first part we only have to prove that  $X_{\CC}(p, q)\supset (Y(p, q)\cap\CC)$. Let $m = p\ast_{h} q\in\CC$ for some $(t(p), t(q))$-mixing partition $h$. Then, $m$ is projective and $m\preceq p\otimes q$ by Lemma \ref{LemTensorSubproj}. Now, let $p'\preceq p$ be a projective partition in $\CC$, such that $m\preceq p'\otimes q$ but $m\not\preceq l\otimes q$ for any $l\preceq p'$, $l\neq p'$, $l\in \CC$. Similarly, let $q'\preceq q$ be a projective partition in $\CC$, such that $m\preceq p'\otimes q'$ but $m\not\preceq p'\otimes r$ for any $r\preceq q'$, $r\neq q'$, $r\in \CC$. Then, $m\in X_{\CC}(p', q')$ and hence $m = p'\ast_{h'} q'$ for some mixing partition $h'$ by Proposition \ref{PropDominatedTensor}. Lemma \ref{LemTensorSubproj} then yields $p = p'$, $q = q'$ and $h = h'$, hence $m\in X_{\CC}(p, q)$.

Let now $m\in X_{\CC}(p, q)$, assume that $P_{m}$ is minimal and set
\begin{equation*}
R = T_{p}\otimes R_{q} + R_{p}\otimes T_{p} - R_{p}\otimes R_{q}
\end{equation*}
so that $P_{p}\otimes P_{q} = T_{p}\otimes T_{q} - R$. By minimality, $RP_{m}R = \lambda P_{m}$ with $\lambda\in \C$ and $\vert\lambda\vert = \|RP_{m}R\| < 1$ because $P_{m}$ is not dominated by $R$. Setting $V = (P_{p}\otimes P_{q})P_{m}$, we then have
\begin{equation*}
V^{*}V = P_{m} - RP_{m}R = (1-\lambda)P_{m}.
\end{equation*}
Let $\mu$ be a square root of $(1-\lambda)^{-1}$, so that $(\mu V)^{*}(\mu V) = P_{m}$. This means that $\mu V$ is a partial isometry, hence $(\mu V)(\mu V)^{*} = \vert\mu\vert^{2}(P_{p}\otimes P_{q})P_{m}(P_{p}\otimes P_{q})$ is a projection dominated by $P_{p}\otimes P_{q}$ and equivalent to $P_{m}$. Letting $v_{m}$ be the associated representation, we can conclude thanks to Lemma \ref{LemFusion}.
\end{proof}

Note that the sum in the above theorem is not a decomposition into orthogonal summands. If $P_{m}$ is not irreducible, then $RP_{m}R$ is a linear combination of operators of the form $P_{m}T_{\sigma}P_{m}$. There is no obvious reason why this should be a projection in general, so that the fusion rules may be more complicated when crossing are allowed.


\section{Special cases}\label{SecSpecial}

In this section, we study the two extreme cases of easy quantum groups: \emph{classical groups} and \emph{free quantum groups}. In the first case, the linear independence problem is in some sense not tractable. We try to formalize this problem and link it to some algebraic notions from classical representation theory. In the second case, all the non-injectivity issues of the previous section vanish and we can give a complete unified description of the representation theory.

\subsection{Classical groups}

The list of easy (classical) groups includes $O_{N}$ and $S_{N}$. For the complete list see \cite{banica2009liberation}. Their corresponding categories of partitions contain the crossing partition $\crosspart\in P(2, 2)$. Recall from \cite{banica2009liberation} that $O_{N}$ corresponds to the category of all pair partitions $P_{2}$ whereas $S_{N}$ corresponds to all partitions $P$.

To formalize the issue of linear independence for the maps $T_{p}$, we introduce an algebra inspired from the classical representation theory of orthogonal groups. These algebras may be introduced for general categories $\CC$ of partitions (not only those containing the crossing partition $\crosspart$).

\begin{de}
Let $\CC$ be a category of partitions and let $k$ be an integer. The \emph{generalized Brauer algebra} of $\CC$ of order $k$ with parameter $N$ is the vector space generated by all partitions $p\in \CC(k, k)$ together with the algebra structure induced by
\begin{equation*}
q.p = N^{-\rl(q, p)}qp
\end{equation*}
and the involution $p\mapsto p^{*}$. We will denote it $B_{k}(\CC, N)$.
\end{de}

\begin{rem}
The classical representation theory of orthogonal groups involves combinatorial tools called \emph{Brauer diagrams} introduced by R. Brauer in \cite[Sec 5]{brauer1937algebras}. These are in fact exactly pair partitions in $P_{2}(k, k)$. From this remark we see that if $\CC = P_{2}$ is the category of all pair partitions, the algebras $B_{k}(P_{2}, N)$ are the so-called \emph{Brauer algebras} with parameter $N$.
\end{rem}

The easiness assumption then implies that the $*$-homomorphism
\begin{equation*}
\Xi: B_{k}(\CC, N) \rightarrow \Hom(u^{\otimes k}, u^{\otimes k})
\end{equation*}
is surjective. According to \cite[Sec 5]{brauer1937algebras}, this fact is also a consequence of the \emph{first fundamental theorem of invariant theory for $O_{N}$}.

Controlling linear independence now amounts to understanding the kernel $K_{k}(\CC, N)$ of the map $\Xi$. Let us give two basic results concerning this problem:
\begin{itemize}
\item By a simple dimension-counting argument, we see that $K_{k}(\CC, N) = \{0\}$ if $k\leqslant N$. If moreover $\CC$ contains $\crosspart$ (i.e. if $\G$ is a classical group) then $K_{k}(\CC, N) \neq \{0\}$ if $k > N$.
\item The \emph{second fundamental theorem of invariant theory for $O_{N}$} gives an explicit set of generators of the ideal $K_{k}(P_{2}, N)$ (see e.g. \cite{lehrer2012second} for a recent statement of this result).
\end{itemize}

\begin{rem}
G. Lehrer and R. Zhang improved in \cite{lehrer2012second} the second fondamental theorem of invariant theory for $O_{N}$ by giving an explicit idempotent $E$ generating the whole ideal $K_{k}(P_{2}, N)$. This means in particular that the two-sided ideal generated by $E$ in $B_{k}(\CC, N)$ is contained in $K_{k}(\CC, N)$ for any category of partitions $\CC$ containing the simple crossing. It would be interesting to know whether this ideal is always the whole kernel or not.
\end{rem}

The study of the generalized Brauer algebras is a problem which is beyond the scope of this paper. It is, however, a necessary step towards the understanding of fusion rules not only for easy classical groups but also for easy quantum groups whose categories involve partitions with some crossings. We should also point out that it is even quite unclear how to link our work with the classical theory of, say, $S_{N}$. Let us give an example. Two projective partitions in $P$ are obviously equivalent if and only if they have the same number of through-blocks. Thus, our general theory gives us a family of representations $u_{k}$ indexed by $\N$ which are either $0$ or non-equivalent to any other. The natural question, to which we have no answer at the moment, is:

\begin{q}
How do these representations decompose into irreducible ones?
\end{q}

For the moment, this "group issue" suggests that the representation theory of easy classical groups should rather be used as a building block to study general easy quantum groups. Giving precise statements concerning the way these building blocks enter the picture is difficult for the moment. However, a work of S. Raum and the second author \cite{raum2013easy} provides strong evidence for this, by giving many new explicit examples.

\subsection{Free quantum groups}

All along this section, we make the assumption $N\geqslant 4$. By lemma \ref{lem:linearindependence}, this implies that the maps $T_{p}$ for $p\in \CC(k, k)$ are linearly independent for any category of noncrossing partitions $\CC$. This is precisely what makes free quantum groups very tractable from the point of view of representation theory, as will appear. Let us first define this notion.

\begin{de}
An easy quantum group $\G$ is said to be \emph{free} if its associated category of partitions $\CC$ is noncrossing.
\end{de}

Free easy quantum groups have been completely classified in \cite{banica2009liberation} and \cite{weber2012classification}. Their representation theory was studied in \cite{banica1996theorie}, \cite{banica1999symmetries} and \cite{banica2009fusion} -- nevertheless, our approach gives a unified treatment of all these results. It also enlights the fact that free quantum groups are in some sense much easier to handle than classical groups. This comes right from the linear independence of the maps $T_{p}$.

\begin{thm}\label{ThmFusionRulesNC}
Let $\CC$ be a category of \emph{noncrossing} partitions, let $N\geqslant 4$ and let $\G$ be the associated easy quantum group. Then, the representation $u_{p}$ is non-zero and irreducible for every projective partition $p\in \CC$. In other words, there is a bijection
\begin{equation*}
\Irr(\G) \simeq \left(\bigcup_{k}\Proj(k)\right)/\sim
\end{equation*}
\end{thm}

\begin{proof}
Linear independence ensures that $u_{p}\neq 0$ and that the map $\Psi$ of Proposition \ref{PropSymP} is an isomorphism. Since the symmetry group of a noncrossing projective partition is trivial (see Remark \ref{RemNCTrivialSym}), we deduce that $u_{p}$ is irreducible.
\end{proof}

If $\CC$ is a category of noncrossing partitions, $\Sym(p, q)$ contains at most one element for every $p$ and $q$. Thus, according to Proposition \ref{PropUnitaryEquivalence}, two representations $u_{p}$ and $u_{q}$ are unitarily equivalent if and only if $r^{p}_{q} = r^{p}_{q}(\id)\in \CC$. Using this, we can recover some well-known facts about the representation theory of free easy quantum groups.

Up to now, our labelling of the equivalence classes of irreducible representations only tells us that there are countably many of them, which is not very interesting. The whole strength of this labelling is that it behaves nicely with respect to the tensor product and Theorem \ref{ThmFusionRules} allows us to recover the fusion rules of free easy quantum groups. To do this, we have to determine the sets $X_{\CC}(p,q)$ for noncrossing categories of partitions $\CC$. It turns out that the only noncrossing mixing partitions (see Definition \ref{DeMixPart}) are of the following form.

We denote by $h_{\square}^{k}$ the projective partition in $NC(2k, 2k)$ where the $i$-th point in each row is connected to the $(2k-i+1)$-th point in the same row (i.e. an increasing inclusion of $k$ blocks of size $2$). If moreover we connect the points $1$, $k$, $1'$ and $k'$, we obtain another projective partition in $NC(2k, 2k)$ denoted $h_{\boxvert}^{k}$:
\setlength{\unitlength}{0.5cm}
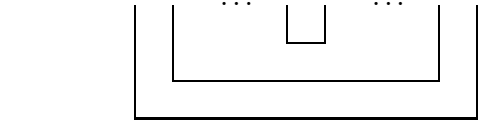
\begin{center}
\begin{picture}(15,9)
 \put(1,4){$h_{\square}^{k} =$}
 \put(3,9){\partii{3}{1}{10}}
 \put(3,9){\partii{2}{2}{9}}
 \put(3,9){\partii{1}{5}{6}}
 \put(6.5,8){$\ldots$}
 \put(10.5,8){$\ldots$}
 \put(6.5,0){$\ldots$}
 \put(10.5,0){$\ldots$}
 \put(3,0){\uppartii{3}{1}{10}}
 \put(3,0){\uppartii{2}{2}{9}}
 \put(3,0){\uppartii{1}{5}{6}}
\end{picture}
\end{center}
and
\setlength{\unitlength}{0.5cm}
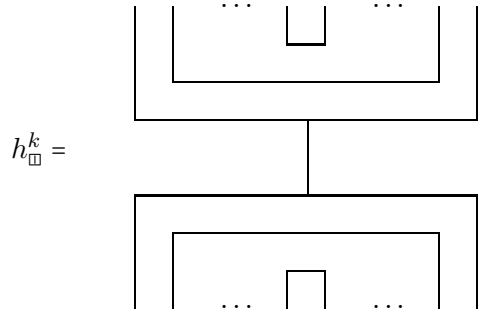
\begin{center}
\begin{picture}(15,9)
 \put(1,4){$h_{\boxvert}^{k} =$}
 \put(3,9){\partii{3}{1}{10}}
 \put(3,9){\partii{2}{2}{9}}
 \put(3,9){\partii{1}{5}{6}}
 \put(6.5,8){$\ldots$}
 \put(10.5,8){$\ldots$}
 \put(6.5,0){$\ldots$}
 \put(10.5,0){$\ldots$}
 \put(3,0){\uppartii{3}{1}{10}}
 \put(3,0){\uppartii{2}{2}{9}}
 \put(3,0){\uppartii{1}{5}{6}}
 \put(8.8,3){\line(0,1){2}}
\end{picture}
\end{center}

As a technical tool, we define the following versions of $h_{\square}^{k}$ and $h_{\boxvert}^{k}$, where we fill up with identity partitions to the left and right. If $\alpha, \beta\in\N_0$ and $k\leq \min(\alpha,\beta)$, we set
\begin{eqnarray*}
{\tilde h}_{\square}^{k}:= \idpart^{\otimes (\alpha-k)}\otimes h_{\square}^{k} \otimes \idpart^{\otimes (\beta - k)}\\
{\tilde h}_{\boxvert}^{k}:= \idpart^{\otimes (\alpha-k)}\otimes h_{\boxvert}^{k} \otimes \idpart^{\otimes (\beta - k)}
\end{eqnarray*}

\begin{de}\label{DefBlockReduction}
Let $p$ and $q$ be projective partitions. For $0\leqslant k\leqslant \min(t(p), t(q))$, we set (referring to Definition \ref{DeMixPartAndPQ}):
\begin{eqnarray*}
p\square^{k} q & := \; p\ast_{{\tilde h}_{\square}^{k}}q & = \; (p_u^*\otimes q_u^*){\tilde h}_{\square}^{k}(p_u\otimes q_u)\\
p\boxvert^{k} q & := \; p\ast_{{\tilde h}_{\boxvert}^{k}}q & = \; (p_u^*\otimes q_u^*){\tilde h}_{\boxvert}^{k}(p_u\otimes q_u)
\end{eqnarray*}
\end{de}

In other words, the partition $p\square^{k} q$ is constructed by inserting the partition $h_{\square}^{k}$ (filled up with identity partitions) as a pattern of connecting the through-blocks of $p$ and $q$; likewise $p\boxvert^k q$.

We then have the following statement.

\begin{lem}\label{LemDescriptionOfX}
The only noncrossing mixing partitions are the partitions ${\tilde h}_{\square}^{k}$ and ${\tilde h}_{\boxvert}^{k}$. Thus, for any noncrossing category $\CC$ we have the equality:
\begin{equation*}
X_{\CC}(p, q) = \{p\square^{a}q, p\boxvert^{b}q \;\vert\; 0\leqslant a \leqslant \min(t(p), t(q))\text{ and }1\leqslant b\leqslant \min(t(p), t(q))\}\cap \CC
\end{equation*}
Moreover, if $p$ and $q$ are projective partitions and if $a\leqslant \min(t(p), t(q))$, then
\begin{equation*}
t(p \boxvert^{a} q) + 2a = t(p) + t(q) + 1 \text{ and } t(p \square^{a} q) + 2a = t(p) + t(q).
\end{equation*}
\end{lem}

\begin{proof}
The first part is obvious. The number of through-blocks of $p$ can be written as $t(p) = x + a$, whereas $t(q) = y + a$, for some numbers $x$ and $y$. Furthermore, $t(p \boxvert^{a} q) = x + y + 1$. This yields the first part of the equation, and likewise for the case of $p\square^{a} q$.
\end{proof}

\begin{rem}
Note that in particular, any two subrepresentation of $u_{p}\otimes u_{q}$ have different number of through-blocks, hence they are orthogonal (see the proof of Proposition \ref{ThmGeneralDecompositionEquivClasses} in combination with Lemma \ref{LemEquivAndT}). Thus, the sums in the formula of Theorem \ref{ThmFusionRules} become true orthogonal sums of representations.
\end{rem}

With this simpler formulation we can recover straightforwardly the representation theory of free easy quantum groups. For the convenience of the reader, we write it in such a way such that it could be used as a starting point to read this article.

\begin{ex}
The \emph{free symmetric quantum group} $S_{N}^{+}$ was introduced by S. Wang in \cite{wang1995free} and its representation theory was studied by T. Banica in \cite{banica1999symmetries}. 
Since the intertwiner spaces $\Hom(u^{\otimes k},u^{\otimes l})$ of $S_N^+$ are spanned by the maps $T_p$, where $p\in NC(k,l)$, the quantum group $S_N^+$ is easy, with corresponding category $NC$ (all noncrossing partitions). (See \cite{banica2009liberation} and \cite{weber2012classification} for details.) We can now give a new way of computing the fusion rules for $S_N^+$.

To any projective partition $p\in NC(k,k)$, we associate the projection $P_p=T_p-R_p\in \Hom(u^{\otimes k},u^{\otimes k})$ and the representation $u_p=(\iota\otimes P_p)(u^{\otimes k})$ according to Definition \ref{DeSubreprAssocToPartitions}. By Theorem \ref{ThmFusionRulesNC}, we infer that $u_p$ is nonzero and irreducible, if $N\geq 4$. By Theorem \ref{ThmUnitaryEquivalence}, we know that for two projective partitions $p,q\in NC$, the representations $u_p$ and $u_q$ are unitarily equivalent if and only if there exists a partition $r\in NC$ such that $r^*r=p$ and $rr^*=q$. If this is the case, we have $t(p)=t(r^*r)=t(rr^*)=t(q)$ for the number of through-blocks (see Lemma \ref{LemPPStarProj}). Conversely, if $t(p)=t(q)$, we consider the through-block decompositions $p=p_u^*p_u$ and $q=q_u^*q_u$ of $p$ and $q$ according to Proposition \ref{PropThroughBlockDecomp}. The partition $r:=q_u^*p_u$ is in $NC$ and we infer that $r^*r=p$ and $rr^*=q$, using Lemma \ref{LemPPStar}.

Therefore, two representations $u_p$ and $u_q$ are unitarily equivalent if and only if $t(p)=t(q)$. Hence, we can label the irreducible partitions of $S_N^+$ by positive integers $\N_0$. Note that to any $k\geq 1$, the partition $p=\idpart^{\otimes k}\in NC$ fulfills $t(p)=k$ and hence we can choose it as a representative for $u_k$. For $k=0$, the partition $p_0=\{1\}\{1'\}\in NC(1,1)$ given by a singleton on the upper point and a singleton on the lower point is a representative for $u_0$, since $t(p_0)=0$.
From Theorem \ref{ThmFusionRules}, we now obtain the fusion rules:
\[u_k\otimes u_l = \sum_{m\in X_{NC}(p,q)} u_m\]
Here, $p$ and $q$ are the canonical representatives for $k,l\in \N_0$. By Lemma \ref{LemDescriptionOfX}, the set $X_{NC}(p,q)$ is given by:
\[X_{NC}(p,q)=\{p\square^a q, p\boxvert^b q\;|\; 0\leq a\leq \min(k,l), 1\leq b\leq \min(k,l)\}\]
Furthermore, $t(p\square^a q)=t(p)+t(q)-2a=k+l-2a$ and $t(p\boxvert^b q)=k+l-2b+1$, thus the values of $t(m)$ for $m\in X_{NC}(p,q)$ are all numbers $|k-l|\leq x\leq k+l$. This yields the fusion rules for $S_N^+$:
\[u_k\otimes u_l = u_{|k-l|}\oplus u_{|k-l|+1}\oplus\ldots\oplus u_{k+l}\]
\end{ex}

\begin{ex}
Another, and in fact older, example of a free easy quantum group is the \emph{free orthogonal quantum group} introduced by S. Wang in \cite{wang1998quantum}. Its associated category of partitions is the set $NC_{2}$ of all noncrossing pair partitions and its representation theory was first computed by T. Banica in \cite{banica1996theorie}. 

As in the case of $S_{N}^{+}$, two projective partitions $p$ and $q$ in $NC_2$ yield equivalent representations if and only if $t(p) = t(q)$. (Note that $p_u$ is a noncrossing \emph{pair} partition, if $p$ is a noncrossing pair partition.) Thus, we can label the irreducible representations $u_p$, $p\in NC_2$ by positive integers $k\in\N_0$, choosing representatives $p\in NC_2$ with $t(p)=k$.

Since $p\boxvert^b q\notin NC_2$, the set $X_{NC_2}(p,q)$ is only given by
\[X_{NC_2}(p,q)=\{p\square^a q \;|\; 0\leq a\leq \min(k,l)\}\]
and the values of $t(p\square^a q)$ are all numbers $|k-l|, |k-l|+2,\ldots,k+l$.
Theorem \ref{ThmFusionRules} then yields:
\[u_k\otimes u_l = u_{|k-l|}\oplus u_{|k-l|+2}\oplus\ldots\oplus u_{k+l}\]
\end{ex}

\begin{ex}
The \emph{modified symmetric quantum group} $S_N'^+$, the \emph{modified bistochastic quantum group} $B_N'^+$ and the \emph{freely modified bistochastic quantum group} ${B_N^\#}^+$ (see \cite{banica2009liberation} and \cite{weber2012classification}) are \emph{not} closed under the parts of the through-block decomposition. For instance, the partition $p_0=\{1\}\{1'\}\in NC(1,1)$ is in all categories $\CC$ corresponding to $S_N'^+$,  $B_N'^+$ and ${B_N^\#}^+$ respectively, but its through-block decomposition $p_0=p_u^*p_u$ yields the singleton $p_u=\singleton$, which is not in $\CC$ in all three cases. Therefore, the classes of irreducible representations are not given only by the number of through-blocks -- a circumstance which also appears in the next example.
\end{ex}

Let us now turn to a more difficult example, the  \emph{hyperoctahedral quantum group} $H_N^+$ introduced by Banica, Bichon and Collins in \cite{banica2007hyperoctahedral}. The computation of the fusion rules of this easy quantum group is a bit more tricky than for $S_N^+$ and $O_N^+$. The reason is exactly the same as in the previous example: The category $\langle\vierpart\rangle$ of all noncrossing partitions with blocks of even size -- corresponding to $H_N^+$ -- is \emph{not} closed under the through-block decomposition. 
Consider for instance the partition $p=\vierpartrot\in NC(2,2)$ whose through-block decomposition yields the partition $p_u\in NC(2,1)$ consisting of a single three block -- which is not in $\langle\vierpart\rangle$. 

Thus, we have irreducible representations $u_p$ and $u_q$ which are not isomorphic, but fulfill $t(p)=t(q)$. As an example, take the projective partitions
\begin{equation*}
p = \vierpartrot\idpart \text{ and } q = \idpart\vierpartrot
\end{equation*}
from Remark \ref{RemHN+}.

We hence have to assign more data to a partition $p\in\langle\vierpart\rangle$ in order to determine the equivalence classes of the irreducible representations $u_p$. If $p$ is such a partition, we associate to each of its through-blocks the number 0 if the block has size $4k$ for some integer $k\in\N$ and 1 otherwise. Looking at through-blocks of $p$ from left to right, we get a word $w(p)$ in the free monoid over $\Z_2$. For instance, the above partition $p$ would yield $w(p)=01$, whereas $w(q)=10$.

\begin{lem}
If $\CC = \langle \vierpart\rangle$, then $p\sim q$ if and only $w(p) = w(q)$. In other words, if $A$ denotes the free monoid over $\Z_{2}$, then two irreducible representations of $H_{N}^{+}$ are unitarily equivalent if and only the associated partitions yield the same word in $A$.
\end{lem}

\begin{proof}
Let $p$ and $q$ be two projective partitions in $\CC$ and consider the partition $r^{p}_{q}$. Assume that $w(p) \neq w(q)$ and let $i$ be the number of the first letter which differs in the two words (starting from the left). Combining the $i$-th through-block of $p_{u}$ with the $i$-th through-block of $q_{u}^{*}$ then yields a block with odd size, and $r^{p}_{q}\notin \CC$. Assume now that $w(p) = w(q)$. Then, $r^{p}_{q}$ is noncrossing and every through-block is of even size. Since non-through-blocks of $r^{p}_{q}$ come from non-through-blocks of $p$ and $q$ which are all noncrossing and of even size, $r^{p}_{q}\in \CC$ and $p\sim q$.
\end{proof}

Thus, the irreducible representations of $H_{N}^{+}$ can naturally be indexed by $A$ and the labelling coincides with that of \cite[Thm 7.3]{banica2009fusion}.

Now, let $p$ and $q$ be projective partitions and assume that $p\square^{k} q \in \langle\vierpart\rangle$. This implies in particular that the through-blocks have even size, hence the $i$-th through-block of $p$, starting from the right, has the same parity as the $i$-th through-block of $q$, starting from the left, for every $1\leqslant i\leqslant k$. This is easier to describe at the level of $A$.

\begin{de}
We denote by $w\mapsto \overline{w}$ be the involution on $A$ consisting in reversing the words.
\end{de}

Then, $p\square^{k} q \in \langle\vierpart\rangle$ if and only if $w(p) = az$ and $w(q) = \overline{z}b$ for a word $z$ of length $k$. Moreover, all through-blocks belonging to $z$ being turned into non-through-blocks under this operation, we have $w(p\square^{k} q) = ab$.

Assume now that $p\boxvert^{k} q \in \langle\vierpart\rangle$. Using the argument above, we see that there is a word $z$ of length $k-1$ such that $w(p) = az$ and $w(q) = \overline{z}b$. Moreover, the $k$-th through-block of $p$, starting from the right, is combined with the $k$-th through-block of $q$, starting from the left. Again, this translates to $A$.

\begin{de}
We denote by
\begin{equation*}
(w, w')\mapsto w\ast w'
\end{equation*}
the operation consisting in adding the last letter of $w$ to the first letter of $w'$ (in $\Z_{2}$) and then concatenating the remainder of the words.
\end{de}

Then, $w(p\boxvert^{k} q) = a\ast b$. Summing up, we have the following formula:
\begin{equation*}
u_{w}\otimes u_{w'} = \sum_{w = az, w' = \overline{z}b} u_{ab} \oplus u_{a\ast b}.
\end{equation*}

\begin{rem}
The fusion rules for $H_{N}^{+}$ were computed by T. Banica and R. Vergnioux in \cite{banica2009fusion} as a special case of the fusion rules for the \emph{quantum reflection groups} $H_{N}^{s+}$. Nevertheless, we believe that our proof above has the advantage of showing the similarity with other free orthogonal quantum groups, whereas their  proof rather uses techniques from unitary easy quantum groups (see Section \ref{SecUnitary}).
\end{rem}

It appears in these examples that the operations $\square^{k}$ and $\boxvert^{k}$ translate into a \emph{concatenation} and a \emph{fusion} operation on the fusion semiring. These operations on the fusion semiring were considered by T. Banica and R. Vergnioux in \cite[Sec 10]{banica2009fusion} in order to formulate some general conjectures on free quantum groups (see Subsection \ref{SubsecConjecture}).

\section{Unitary quantum groups and the freeness conjecture}\label{SecUnitary}

In what precedes, the expression "easy quantum group" was used in a restrictive sense since we only considered \emph{orthogonal} easy quantum groups, i.e. easy quantum groups $\G$ satisfying
\begin{equation*}
S_{N} \subset \G \subset O_{N}^{+}.
\end{equation*}
One can also consider \emph{unitary} easy quantum groups by looking at quantum groups
\begin{equation*}
S_{N} \subset \G \subset U_{N}^{+}
\end{equation*}
satisfying some proper generalization of the "easiness" property. The definition of these unitary easy quantum groups has not appeared yet in the literature, but it is well-known to experts in the field. Our purpose is not to develop the general theory since a comprehensive study of these objects is currently undergone by P. Tarrago and the second author \cite{tarrago2015unitary}. We will only give a quick description of the easiness condition in order to show how our results can be extended to this setting.

\subsection{Colored diagrams}

The difference between $O_{N}^{+}$ and $U_{N}^{+}$ is that in the latter, the fundamental representation $u$ is not equivalent to its contragredient representation $\overline{u}$. Thus, we cannot only look at the intertwiner spaces between $u^{\otimes k}$ and $u^{\otimes l}$ to recover the whole structure of the quantum group. More precisely, we again denote by $A$ the free monoid on the set $\Z_{2}$ and consider two words $w$ and $w'$ of length $k$ and $l$ respectively. We set $u^{0} = u$, $u^{1} = \overline{u}$ and $u^{w} = u^{w_{1}}\otimes \dots \otimes u^{w_{k}}$. We want the space
\begin{equation*}
\Hom(u^{w}, u^{w'})
\end{equation*}
to be spanned by operators associated with partitions. In order to include all these spaces in a single description, we will use \emph{colored partitions}.

\begin{de}
A \emph{(two-)colored partition} is a partition with the additional data of a color (black or white) for each point. The set of all colored partitions is denoted $P^{\circ, \bullet}$.
\end{de}

The usual operations on partitions can be extended to the colored setting.

\begin{itemize}
\item To any colored partition $p$ and any integer $N$, one associates the linear map $T_{p}$ as in Definition \ref{DefTp}. Note in particular that $T_{p}$ \emph{does not depend on the coloring of $p$}.
\item Two colored partitions $q\in P^{\circ, \bullet}(k, l)$ and $p\in P^{\circ, \bullet}(l, m)$ can be composed only if the colors of the lower points of $p$ coincide with the colors of the upper points of $q$.
\item The tensor product of colored partitions is defined in the obvious way, as well as the involution.
\item There is, however, a subtlety concerning the rotation. If a point is rotated from a row to another, then its color is changed. This reflects the passage from $u$ to $\overline{u}$ in Frobenius reciprocity (which is precisely the operation encoded by the rotation).
\item There are four colored identity partitions, but only two of them will be important: the \emph{white identity} partition with both points colored in white and the \emph{black identity} partition with both points colored in black. Note that we can pass from one to the other using the rotation operation.
\end{itemize}

\begin{de}
A \emph{category of colored partitions} is the data of a set $\CC^{\circ, \bullet}(k, l)$ of colored partitions  for all integers $k$ and $l$ which is stable under the above category operations and contains the white identity (hence also the black identity).
\end{de}

Throughout this section, we will use the obvious bijection between coloring on $k$ points and words of length $k$ in $A$ obtained by sending "white" to $0$ and "black" to $1$.

\begin{de}
Let $\CC^{\circ, \bullet}$ be a category of colored partitions and let $N$ be an integer. The associated \emph{easy unitary quantum group} is the unique compact quantum group $\G$ with fundamental representation $u$ such that for any $w, w'\in A$ of length respectively $k$ and $l$, $\Hom(u^{w}, u^{w'})$ is spanned by the operators $T_{p}$ for $p\in \CC^{\circ, \bullet}(k, l)$ such that the upper coloring of $p$ is $w$ and the lower coloring of $p$ is $w'$.
\end{de}

Before trying to extend our results to this setting, we first have to check that the techniques of Section \ref{SecPartitions} still work. Everything relies on the through-block decomposition, thus we only have to make sense of it in the colored context.

\begin{de}
A \emph{colored building partition} is a building partition with any coloring on the upper row and only white points on the lower row. A \emph{colored through-partition} is a through-partition with all points colored in white.
\end{de}

With these definitions, the existence and uniqueness of the "colored through-block decomposition" is clear.

\subsection{Representation theory}

Adapting the results of Section \ref{SecRepresentation} to colored diagrams is in fact straightforward. The essential remark is that as soon as we fix a word $w\in A$ and consider $\Hom(u^{w}, u^{w})$, we have fixed the coloring of all partitions involved. Thus, everything boils down to the non-colored case.

If $p$ is a projective colored partition (i.e. $pp = p = p^{*}$), we define a projection
\begin{equation*}
P_{p} = T_{p} - \bigvee_{q\prec p}T_{q}
\end{equation*}
and a representation $u_{p} = (\ii\otimes P_{p})(u^{\otimes k})$. If $w\in A$, let us denote by $\Projc(w)$ the set of projective partitions in $\CC^{\circ, \bullet}$ the upper (hence also lower) coloring of which is $w$. With this notation, we get the decomposition result.

\begin{prop}
Let $\CC^{\circ, \bullet}$ be a category of colored partitions, let $N$ be an integer and let $\G$ be the associated unitary easy quantum group. Then, for every $w\in A$,
\begin{equation*}
u^{w} = \sum_{p\in \Projc(w)}u_{p}.
\end{equation*}
Again, this sum means that if a subrepresentation of $u^{w}$ contains $u_{p}$ for every $p\in \Projc(w)$, then it is equal to $u^{w}$.
\end{prop}

If $p$ is any projective colored partition, we define its \emph{symmetry group} $\Symc(p)$ to be the set of all through-partitions one can add in the middle of $p$. Again, since the proof of Proposition \ref{PropSymP} only deals with partitions of the form $pqp$ (which have the same coloring as $p$), we get:

\begin{prop}
Let $p$ be a projective colored partition. Then, there is a surjective $*$-homomorphism
\begin{equation*}
\Psi^{\circ, \bullet}: \C[\Symc(p)] \rightarrow \Aut(u_{p}).
\end{equation*}
\end{prop}

The proof of Theorem \ref{ThmUnitaryEquivalence} clearly works in the colored case, giving:

\begin{thm}
Let $\CC^{\circ, \bullet}$ be a category of colored partitions, let $N$ be an integer and let $\G$ be the associated unitary easy quantum group. Let $p\in \Projc(k)$ and $q\in \Projc(l)$. Then, the representations $u_{p}$ and $u_{q}$ are unitarily equivalent if and only if either there exists $r\in \CC^{\circ, \bullet}(k, l)$ such that $r^{*}r = p$ and $rr^{*} = q$ or $u_{p} = u_{q} = 0$.
\end{thm}

For the fusion rules, we again only have to deal with partitions dominated by $p\otimes q$, hence the coloring is fixed. Thus, a straightforward generalization of Proposition \ref{PropDominatedTensor} holds. We want to describe the set $X_{\CC^{\circ, \bullet}}(p, q)$ of projective colored partitions such that 
\begin{equation*}
u_{p}\otimes u_{q} = \sum_{m\in X_{\CC^{\circ, \bullet}}(p, q)}u_{m}.
\end{equation*}
Let $Y^{\circ, \bullet}(p, q)$ be the set of all colored partitions of the form $p\ast_{h} q$ for some $(t(p), t(q))$-mixing partition $h$ (with all points colored in white).

\begin{thm}\label{ThmFusionUnitary}
Let $\CC^{\circ, \bullet}$ be a category of colored partitions, let $N$ be an integer and let $\G$ be the associated unitary easy quantum group. Then, for any projective colored partitions $p$ and $q$, we have
\begin{equation*}
u_{p}\otimes u_{q} = \sum_{m\in X_{\CC^{\circ, \bullet}}(p, q)}u_{m}.
\end{equation*}
with
\begin{equation*}
X_{\CC^{\circ, \bullet}}(p, q) = Y^{\circ, \bullet}(p, q) \cap \CC^{\circ, \bullet}.
\end{equation*}
\end{thm}

\begin{proof}
As before, the inclusion $(Y^{\circ, \bullet}(p, q)) \cap \CC^{\circ, \bullet} \subset X_{\CC^{\circ, \bullet}}(p, q)$ is clear. To prove the converse one, simply note that by symmetry of $p\otimes q$, one can use unicolored identity partitions in the proof of Lemma \ref{LemTensorSubproj}. From this, we see that if $l\ast_{h}r\in \CC^{\circ, \bullet}$, then $l, r\in \CC^{\circ, \bullet}$, which is the only ingredient we need to finish the proof.
\end{proof}

\subsection{An example}

As already mentioned in the beginning of this section, the study and classification of unitary easy quantum groups is at its earliest stage. We will therefore only illustrate our results with the most simple example: the free unitary quantum group.

Let $\CC^{\circ, \bullet}$ be a category of colored \emph{noncrossing} partitions. Note that for a fixed word $w\in A$, the linear maps $T_{p}$ such that the upper and lower coloring of $p$ are given by $w$ are linearly independent as soon as $N\geqslant 4$ by Lemma \ref{lem:linearindependence}. This means that our techniques completely describe the representation theory of such quantum groups. This is interesting because there are infinitely many non-isomorphic free unitary easy quantum groups (for a fixed $N$), in sharp contrast with the orthogonal case (see \cite{tarrago2015unitary}). For the remainder of this section, we assume that $N\geqslant 4$.

The simplest example is the quantum group $U_{N}^{+}$. The associated category of partitions is $\U^{\circ, \bullet} = \langle \emptyset\rangle$ (i.e. it is generated by the white identity partition). In other words, it consists of pair partitions where the through-blocks have the same color on each end and the non-through-blocks have different colors on each end.

Let us first describe $\Irr(U_{N}^{+})$. If $p$ is a projective colored partition in $\mathcal{U}^{\circ, \bullet}$, we can remove any non-through-block in $p$ without changing the unitary equivalence class. Removing all the non-through-blocks yields a colored partition to which we can naturally associate an element in the monoid $\N\ast \N$: Reading from left to right, if we have $k_{1}
$ white points, then $k_{2}$ black points, then $k_{3}$ white points and so on, we get the element
\begin{equation*}
w(p) = k_{1}\ast k_{2}\ast k_{3}\ast\dots \in \N\ast \N.
\end{equation*}
Reciprocally, any such element gives rise to a unique projective colored partition with no non-through-blocks. Let us consider two elements $w$ and $w'$ in $\N\ast \N$ and the associated projective colored partitions $p$ and $p'$. If the representations $u_{p}$ and $u_{p'}$ are unitarily equivalent, the partition obtained by combining the upper part of $p$ and the lower part of $p'$ must be in $\U^{\circ, \bullet}$. Since the only through-blocks in $\U^{\circ, \bullet}$ are the white identity and the black identity, we infer the equality $w = w'$. We thus have a bijection
\begin{equation*}
\Irr(U_{N}^{+}) \simeq \N\ast \N.
\end{equation*}

For the fusion rules, consider two elements $w$ and $w'$ in $\N\ast \N$ and the associated colored partitions $p$ and $p'$. The colored partition $p\boxvert^{k} p'$ can never belong to $\U^{\circ, \bullet}$ since it contains a block of size four. As for the partition $p\square^{k} p'$, it lies in $\U^{\circ, \bullet}$ if and only if each of its non-through-block has endpoints of different colors. This means that the $i$-th upper point of $p$, starting from the right, has color opposite to that of the $i$-th upper point of $p'$, starting from the left for any $1\leqslant i\leqslant k$.

\begin{de}
We denote by $w\mapsto \overline{w}$ the unique antimultiplicative map on $\N\ast\N$ exchanging the two generators.
\end{de}

Then, $p\square^{k} p'$ is in $\U^{\circ, \bullet}$ if and only if there is a decomposition $p = az$ and $p' = \overline{z}b$ and the resulting partition is obviously equivalent to $ab$. Thus, we recover the main result of \cite{banica1997groupe}:
\begin{equation*}
u_{w}\otimes u_{w'} = \sum_{w = az, w' = \overline{z}b} u_{ab}.
\end{equation*}

\begin{rem}
One could use the same ideas to study the \emph{quantum reflexion groups} $H_{N}^{s+}$ of \cite{banica2009fusion} and get an alternative proof of the fusion rules. Let us also mention the free complexifications of free orthogonal easy quantum groups studied by S. Raum in \cite{raum2012isomorphisms}.
\end{rem}

\subsection{The freeness conjecture}\label{SubsecConjecture}

Beyond these examples, our techniques can help making the freeness conjecture of \cite[Sec 10]{banica2009fusion} more precise. Let us introduce some notions to state this conjecture. Let $S$ be a set together with an involution $x\mapsto \overline{x}$ and a binary law
\begin{equation*}
\begin{array}{ccc}
S\times S & \rightarrow & S\cup\emptyset \\
(x, y) & \mapsto & x\ast y
\end{array}
\end{equation*}
called the \emph{fusion} operation. If $R(S)$ is the free monoid on $S$, we can extend the previous operations in the following way: If $w = w_{1}\dots w_{n}$ and $w' = w'_{1}\dots w'_{n'}$ are words in $R(S)$, then
\begin{equation*}
\begin{array}{ccc}
\overline{w} & = & \overline{w}_{n}\dots \overline{w}_{1} \\
w\ast w' & = & w_{1}\dots (w_{n}\ast w'_{1})\dots w'_{n'}
\end{array}
\end{equation*}
By convention, $w\ast w' = \emptyset$ if $w_{n}\ast w'_{1} = \emptyset$ or if one of the two words is empty. This structure enables us to define a tensor product on $R(S)$ in the following way:
\begin{equation*}
w\otimes w' = \sum_{w = az, w' = \overline{z}b} ab + a\ast b.
\end{equation*}
The tensor product and the involution turn $(R^{+}(S), -, \otimes)$ into a \emph{free fusion semiring}. In our terminology, the conjecture can be stated as follows:

\begin{conj}[Banica--Vergnioux]
Let $\G$ be a free unitary easy quantum group. Then, there is a set $S$ together with an involution and a binary law such that the fusion semiring $(R^{+}(\G), -, \otimes)$ of $\G$ is isomorphic to the free fusion semiring $(R^{+}(S), -, \otimes)$.
\end{conj}

\begin{rem}
In that form, the conjecture is false (consider for example the symmetrized quantum permutation group $S_{N}^{+}\times \Z_{2}$). One should add extra assumptions in the conjecture, but we do not know which ones. We will come back to this problem in the end.
\end{rem}

That conjecture has up to now only been proven by explicitely computing the fusion rules of certain free easy quantum groups. As already mentioned, our computations directly give the set $S$ with its full structure (for $U_{N}^{+}$, $O_{N}^{+}$, $S_{N}^{+}$ and $H_{N}^{+}$). Thus, there might be a unified proof of this conjecture involving only general considerations on projective colored partitions.

Let $\CC^{\circ, \bullet}$ be a category of noncrossing colored partitions and let $S(\CC^{\circ, \bullet})$ be the set of equivalence classes of projective partitions having only one block. The equivalence class of $p$ will be denoted $[p]$. Let $\overline{p}$ be the partition obtained by changing the colors of all points in $p$ and set
\begin{equation*}
[p]\ast[q] = [p\boxvert q].
\end{equation*}

\begin{rem}
The representation $u_{\overline{p}}$ is the contragredient of the representation $u_{p}$. In fact, by Theorem \ref{ThmFusionUnitary} we know that $u_{p}\otimes u_{\overline{p}}$ contains a one-dimensional representation $u_{q}$ with $q = p\square\overline{p}$. The upper part of this partition is a rotated version of $p$ and therefore belongs to $\CC^{\circ, \bullet}$. This means that $q$ is equivalent to the empty partition, i.e. $u_{q}$ is the trivial representation. Since $u_{p}$ and $u_{\overline{p}}$ are irreducible, they are contragredient to each other.
\end{rem}

An obvious obstruction to the validity of the freeness conjecture is the existence of a non-trivial one-dimensional representation. In fact, such a representation $u$ satisfies $u\otimes \overline{u} = 1$, whereas any element $a\neq 1$ in a free fusion semiring satisfies $aa^{*} \neq 1$ (this was already noticed in \cite[Rem 4.4]{raum2012isomorphisms}). With this in hand, we can state a more precise version of the freeness conjecture:

\begin{conj}
Let $\G$ be a free unitary easy quantum group without any non-trivial one-dimensional representation. Then, there is an isomorphism of fusion semirings
\begin{equation*}
(R^{+}(\G), -, \otimes) \simeq (R^{+}(S(\CC^{\circ, \bullet})), -, \otimes).
\end{equation*}
\end{conj}

This form has several advantages:
\begin{itemize}
\item It allows to test directly the conjecture by computing the representation theory.
\item It may produce natural restrictions on the quantum group $\G$ leading to the right characterization of free quantum groups satisfying the easiness conjecture.
\end{itemize}

The general strategy goes through the study of the natural map
\begin{equation*}
\Phi: R^{+}(S(\CC^{\circ, \bullet})) \rightarrow R^{+}(\G)
\end{equation*}
and proving that, under suitable assumptions, it is a bijection preserving the tensor product.
The tools to prove such statements are typically those used to classify categories of partitions. Hence, progress in the classification of colored partitions will probably give further evidence for the freeness conjecture. To conclude, let us mention a companion conjecture which shows the interest of the free structure from the point of view of operator algebras.

\begin{conj}[Banica--Vergnioux]
Let $\G$ be a compact quantum group satisfying the freeness conjecture. Then, the reduced C*-algebra of $\G$ is simple with unique trace.
\end{conj}

\bibliographystyle{amsplain}
\bibliography{../../quantum}

\end{document}